\documentclass[onethmnum,onetabnum,oneeqnum,onefignum]{siamart220329}

    \usepackage[T1]{fontenc}
    \usepackage[utf8]{inputenc}

    \usepackage{amsfonts,amssymb}

    \usepackage[pdftex]{pict2e} 

    \usepackage{enumitem}  

    \newsiamremark{remark}{Remark}
    \newsiamremark{example}{Example}

    \crefname{subsection}{Section}{Sections}
    \Crefname{subsection}{Section}{Sections}
    \crefname{section}{Section}{Sections}

\usepackage[textsize=footnotesize]{todonotes}
  \makeatletter
  \@mparswitchfalse%
  \makeatother
  \normalmarginpar

\newcommand\RR{\mathbb{R}}
\newcommand\ZZ{\mathbb{Z}}
\newcommand\NN{\mathbb{N}}
\newcommand\Sun{\mathbb{S}^1}

\newcommand{\vecfield}{X}
\newcommand{\manifM}{M}
\newcommand{\bigmanif}{\mathcal{M}}
\newcommand{\ProjLin}{P}
\newcommand{\demi}{{\textstyle\frac12}}

\newcommand{\deltax}{\delta\hspace{-.25ex}x}
\newcommand{\deltau}{\delta\hspace{-.19ex}u}
\newcommand{\deltaI}{\delta\hspace{-.21ex}I}
\newcommand{\Acc}{\mathcal{A}}
\newcommand{\AccFam}{\mathcal{A}}
\newcommand{\AccLin}{\mathsf{A}}
\newcommand{\tf}{t_{\!f}}

\newcommand{\LL}{L}

\newcommand{\BAcc}{\mathcal{B}}

\DeclareMathOperator{\co}{conv}
\DeclareMathOperator{\cone}{cone}

\DeclareMathOperator{\vect}{Span}

\title{On the controllability of \\nonlinear systems with a periodic drift}
\author{J.-B. Caillau\thanks{Université Côte d'Azur, CNRS,
    LJAD, Inria, France
    (\email{jean-baptiste.caillau@univ-cotedazur.fr}, \email{alesia.herasimenka@univ-cotedazur.fr}).}
\and L. Dell'Elce\thanks{Université Côte d'Azur, Inria, CNRS, LJAD, France
    (\email{lamberto.dell-elce@inria.fr},
    \email{jean-baptiste.pomet@inria.fr}).}
\and A. Herasimenka\footnotemark[1]
\and J.-B. Pomet\footnotemark[2]
}
\date{October 28, 2024}

\usepackage[normalem]{ulem}

\begin{document}
\maketitle 

\begin{abstract}
Sufficient and necessary conditions are established for controllability of affine control systems
where the control is constrained to a set whose convex hull contains the origin but is not necessarily,
in contrast with previously known results,
  a neighborhood of the origin.
  Part of the results, in particular these on global controllability,
  are specific to systems where all solutions with zero control (drift)
are periodic.
These conditions are expressed by means of pushforwards along the flow
of the drift, rather than in terms of Lie brackets; it turns out that they amount to  
local controllability of a time-varying linear approximation with constrained controls.
Global and local results are given, as well as a few illustrative examples.
\end{abstract}



\section*{Introduction}
\noindent This paper is devoted to controllability properties of affine
control systems, defined by a drift vector field and $m\geq1$ control vector fields.
We consider either local controllability around a solution with zero control
(solution of the drift vector field) or global controllability in the case where all such
solutions are periodic; in both cases, the control is constrained to a subset $U$ of
$\RR^m$ that contains zero, but whose convex hull, unlike in most known controllability
results, is not necessarily a neighborhood of zero.
Our original motivation for this study was the control of spacecrafts propelled by solar sails, see references
\cite{ecc-2022,JGCD}\footnote{A preliminary version of some results
  stated in \cref{sec-jurdj-results} can be found in \cite[Section II]{ecc-2022}.
by the authors.}
Governed by solar radiation pressure, solar sails are only
capable of generating forces
contained in a subset $U$ of a convex cone of revolution around the direction
of the incoming light. Zero is the vertex of this cone and is
  contained in $U$, but is obviously not in the interior of the convex
  hull of $U$.
If we consider an unperturbed controlled Kepler problem in the region of negative total energy, all solutions of the drift are periodic.

Devising conditions on the vector fields (and the set $U$) for controllability is quite an old
problem.
Most known sufficient conditions for controllability assume that the
Lie brackets of the vector fields defining the system span the tangent space 
    at all points (bracket generating property),
and prove controllability under additional conditions concerning in general the drift vector
field,
as recalled in \cref{section:state-of-the-art}, based on
  \cite{Lobr70,Bonn81cras,Suss-Jur72,Kren74,jurdjevic_geometric_1996,agrachev_control_2004}. Nevertheless, these results always require ---at least--- 
that zero belongs to the interior of the convex hull of $U$.
This assumption is crucial
  in the proofs, and allows for a decoupling between this condition on
  $U$ and conditions on the vector fields.
In the present paper, where zero is allowed to belong to the boundary of $\co U$, novel
  results are based on \emph{ad hoc} constructions involving the transportation along the flow of the drift vector field of some
  subsets of the tangent space defined from the control vector fields \emph{and} the set $U$.

In \cref{sec-statement}, we cast precisely the controllability problem we want to
look at, and recall the main definitions and results on the subject.
All results and proofs are stated in \cref{sec-jurdj}, completed by a
short appendix on time-varying linear control systems.
\Cref{sec-jurdj-results} is devoted to
sufficient conditions for global controllability in unprescribed time
under the periodicity assumption on the drift. The given proofs
generalize the well-known ones in \cite{Bonn81cras} (for instance), relying
on augmentation of families of vector fields (exposed into details, \textit{e.g.}, in
the textbook \cite{jurdjevic_geometric_1996}).
In \cref{sec-xbar}, we give sufficient conditions for local
controllability in prescribed time along trajectories of the drift,
periodic or not, via a different method of proof, based on linear
approximation as a control-constrained time-varying linear system.
\Cref{sec-conseq} uses the latter results in the periodic case,
yielding a more precise \emph{local} picture (local controllability in one revolution)
  than \cref{sec-jurdj-results}.
\Cref{sec-obstruc} gives necessary conditions for controllability:
some possible ``shapes''
of $U$ forbid controllability.
\Cref{sec-comments+examples} is devoted to comments on the results,
more general than the ones given along in \cref{sec-jurdj}, and to two
academic examples, one illustrating our results, and the other one
exploring a case where controllability holds while none of our
sufficient conditions is met. 
    
\section{Statement of the problem} \label{sec-statement}

\subsection{Systems}
\label{sec-sys}
We consider an affine control system
\begin{equation}
  \label{sys}
    \dot x=
    \vecfield^0(x)+
             \sum_{k=1}^mu_k\,\vecfield^k(x)\,,\ \
    u=(u_1,\ldots,u_m)\in U\subset\RR^m,
\end{equation}
where the set $U$ defines the constraints on the
control and $\vecfield^0,...,\vecfield^m$ are $m+1$ smooth ($C^\infty$) vector fields on a smooth
connected (Hausdorff, second countable)
manifold $\bigmanif$ of dimension $d$.
We do not require $m\leq d$ although it usually is the case.
We assume that the vector field $\vecfield^0$ is complete.

\paragraph{Assumptions on the control set $U$}
All over the paper, we assume that
\begin{equation}
  \label{eq:assU}
  0\in U \quad\text{ and }\quad \vect U=\RR^m\,.
\end{equation}
The first point is natural: our results partially rely on assumptions on the behavior of the
differential equation $\dot x=\vecfield^0(x)$ obtained for the zero control, that should
hence be permitted.
The second point $\left(\vect U=\RR^m\right)$ is not restrictive: if $\vect U$ was a subspace of $\RR^m$ of dimension
$m'<m$, one could pick $m'$ new control vector fields, constant linear
combinations of $\vecfield^1,\ldots,\vecfield^m$ and use them as
control vector fields in a system with $m'$ controls satisfying the
assumption.

\paragraph{Periodic orbits of the drift, angular variable}
Let us denote by $\exp(t \vecfield)$ the flow at time $t$ of a smooth vector field $X$.
The following property of the drift is assumed in some results on global controllability,
and in the rest of this subsection.
\begin{equation}
  \label{per}
  \begin{array}{l}
    \vecfield^0 \text{ has no equilibrium point, and, for any } x \text{ in } \bigmanif,
    \\
    t\mapsto\exp(t \vecfield^0)(x)\ \text{is periodic with minimal period }T(x).
  \end{array}
\end{equation}
%
Let $\manifM$ be the set of periodic orbits and
$\pi$: $\bigmanif\to\manifM$ the 
projection that maps any $x$ in $\bigmanif$ to its orbit, so that
$\pi^{-1}(\pi(x))=\{\exp(t \vecfield^0)(x),\;t\in\RR\}$.

\begin{proposition}
  \label{prop-pi}
  Assuming \eqref{per}, the set $\manifM$ of periodic orbits of $\vecfield^0$ is a smooth
  manifold of dimension $d-1$ and the projection
  $\pi\!:\bigmanif\to\manifM$
  is a (locally trivial) fibration with
  fiber $\Sun=\RR/2\pi\ZZ$, \textit{i.e.},  any $\bar I\in\manifM$ has a
  neighborhood $O$ in $\manifM$ such that $\pi^{-1}(O)$ is diffeomorphic to
  $O\times\Sun$.
\end{proposition}

\begin{proof} 
    The manifold structure of $\manifM$ is obtained as follows\footnote{
      Alternatively, $M$ is a manifold as the quotient of $\bigmanif$ by the smooth free action
      $(\varphi,x)\mapsto\exp\bigl((T(x)/2\pi)\varphi\vecfield^0\bigr)(x)$
      of the compact Lie group $\Sun\!=\!\RR/2\pi\ZZ$ on $\bigmanif$;
      see \emph{e.g.}, \cite[Theorem 1.95]{Gall-Hul-Laf04}.
    }:
  any orbit $\bar I$ has (many) smooth transverse sections $\Sigma$ (smooth submanifold of
  $\bigmanif$ of dimension $d-1$), small enough to
  intersect any orbit at most once; let $O$ be the set of orbits that intersect $\Sigma$,
  and $\rho:O\to\Sigma$ the one-to-one map that sends such an orbit 
  to the point where it intersects $\Sigma$; $O$ provides a
  neighborhood of $\bar I$ in $\manifM$, that inherits the manifold structure of $\Sigma$.
  Concerning the fiber bundle structure, for an arbitrary $\bar I$, take $O$ as above and define $\Phi\!:O\times\Sun\to\bigmanif$ by
  $\Phi(I,\varphi)=\exp\left(\frac{T(\rho(I))}{2\pi}\varphi\,\vecfield^0\right)\left(\rho(I)\right)$, where we
  note that \eqref{per} defines $T(\cdot)$ as smooth as $\vecfield^0$.
\end{proof}

On such an open set $O\times\Sun$, using a diffeomorphism $\Phi:O\times\Sun\to\pi^{-1}(O)$,
and decomposing on $T_{(I,\varphi)}(O\times\Sun)=T_{I}O\oplus
T_\varphi\Sun=T_{I}O\oplus\RR$, one has
\begin{equation}
  \label{Xi-Fi}
  \vecfield^0=\Phi_*\left(\omega\frac\partial{\partial\varphi}\right)
    \ \ \text{and}\ \ 
  \vecfield^i=\Phi_*\left(f^i\frac\partial{\partial\varphi}+F^i\right)=F^i,\ \;i=1,\ldots,m
\end{equation}
where $\omega\!:\manifM\times\Sun\to\RR$ and each $f^k\!:\manifM\times\Sun\to\RR$ are smooth
functions, 
and each $F^k$ is a vector field on $O$ smoothly depending on
$\varphi\in\Sun$ (or a vector field on $O\times\Sun$ with zero component on $T\Sun$),
and the pushforward notation $\Phi_*$ is, as usual, defined by
\begin{equation}
  \label{pushforward}
\Phi_*(X)(y) := \Phi'(\varphi^{-1}(y)) \cdot
X(\Phi^{-1}(y)),\quad y \in \bigmanif.
\end{equation}
for any diffeomorphism $\Phi$ and any vector field $X$ on a manifold $\bigmanif$.
Then System \eqref{sys} can be written as follows to take advantage of the product structure
\begin{equation}
  \label{sys-Iphi}
  \begin{array}{l}
    \displaystyle
    \dot I = 
    \sum_{k=1}^mu_k\,F^k(I,\varphi)\,,
    \\[-1ex]
    \displaystyle
    \!\dot\varphi=\omega(I) 
     +
    \sum_{k=1}^mu_k\,f^k(I,\varphi)\,,
  \end{array}
   \qquad u=(u_1,\ldots,u_m)\in U\,.
 \end{equation}
Note that $\omega(I)$ naturally comes out as $\omega(I,\varphi)$;
composing $\Phi$ with a well chosen diffeomorphism of $O\times\Sun$ of the form
$(I,\varphi)\!\mapsto\!(I,a(I,\varphi))$  removes the dependence  on $\varphi$, so
that $\omega$ can be taken as a function from $\manifM$ to $\RR$.
This product form allows for a simpler and more concrete statement
of results or proofs when System~\eqref{sys} satisfies \eqref{per},
see for instance \eqref{eq:1}--\eqref{eq:26a}.

In the applications to control of solar sails mentioned in the introduction, and
at least without perturbations to a Newtonian potential, $\bigmanif$ (of dimension 6) being
the phase space restricted to negative total energy, Condition \eqref{per} is met. It is customary to use a product form like \eqref{sys-Iphi} where $I$ contains five independent first integrals of the Kepler problem, and $\varphi\in\RR/2\pi\ZZ$ is an angle indicating position on an ellipsis.
It is well known in celestial mechanics that, as we saw under the mere Condition \eqref{per}, the fibration $\bigmanif\to\manifM$ is defined globally, but $\bigmanif$ is not
diffeomorphic to $\manifM\times\Sun$; this is why we do not assume $O=\manifM$ in
\eqref{Xi-Fi}--\eqref{sys-Iphi}, in order to keep an equivalence with the general form
\eqref{sys} under Assumption \eqref{per}.
Our illustrative examples in \cref{sec-comments+examples} and \cref{rmk:cond-suppl-neg}
are however globally in the product form \eqref{sys-Iphi} for simplicity.

\subsection{Accessible sets and controllability, definitions}
\label{sec-def-access}
A solution of \eqref{sys} is a map $t\mapsto(x(t) ,u(t))$,
defined on some interval, valued in $\bigmanif \times U$,
where $u(.)$ is locally essentially bounded and $x(.)$ is absolutely continuous, 
that satisfies \eqref{sys} for almost all time $t$.
For such a solution on an interval $[0,\tf]$, $\tf\geq0$, one says that
the control $u(.)$ steers $x(0)$ to $x(\tf)$ in time $\tf$; indeed
the solution is unique given $u(.)$ and $x(0)$.
One may also say that the same control steers $x(\tf)$ to $x(0)$ in time $-\tf$ because
uniqueness holds also backwards.

\begin{remark}
  \label{rem-ext}
  In the situation where \eqref{per} is satisfied,
  the following fact is true, with $\pi$ as in \cref{prop-pi}:
  if there is a control steering an initial point $x$
  to a final point $y$ in some time $\tau$, then there is also a control
  steering any point $x'$ in the same orbit as $x$ ($\pi(x')=\pi(x)$) to
  any point $y'$ in the same orbit as $y$ ($\pi(y')=\pi(y)$), in time less
  than $\tau+T(x)+T(y)$. Indeed, the control consists, for some $t_1$ in $[0,T(x))$
  and $t_2$ in $[0,T(y))$, in applying a zero control on
  $[0,t_1]$ and $[t_1+\tau, t_1+\tau+t_2]$ 
  and the
  original control on $[t_1, t_1+\tau]$.
\end{remark}


The reachable set from $x_0$ in time $\tf\in\RR$ is the set
of points in $\bigmanif$ that can be reached in forward or backward time $\tf$ from $x_0$ for
some control:
\begin{equation}
  \label{eq:accT}
  \Acc^{U}_{\tf}(x_0)=\{x(\tf),\;
  \text{with }t\mapsto(x(t),u(t))
  \text{ solution of \eqref{sys} on $[0,\tf]$},\;
  x(0)=x_0\}\,.
\end{equation}
We keep the set $U$ as a superscript to stress the constraint on the control.
Obviously, for two subsets $V_1$ and $V_2$ of $U$,
$\Acc^{V_1}_{\tf}(x)\subset\Acc^{V_2}_{\tf}(x)$ if $V_1\subset V_2$.
It is of course assumed that the interval of definition of the solution contains $[0,\tf]$.
If $\tf<0$, $[0,\tf]$ should be understood as $[\tf,0]$, and
it is clear that 
$y$ belongs to $\Acc^{U}_{t}(x)$ if and only if $x$ belongs to $\Acc^{U}_{-t}(y)$. 
From accessible sets in prescribed time, we define
\addtocounter{equation}{1}
\begin{equation}
  \label{eq:acc}
  \Acc^{U}(x)
  =\bigcup_{0\leq t}\Acc^{U}_t(x)
  \quad\text{and}\quad
  \Acc^{U}_{\mathbf{-}}(x)
  =\bigcup_{t\leq 0}\Acc^{U}_t(x)\,,
\end{equation}
the accessible sets in unprescribed forward and backward time
($\Acc^{U}(x)$ is sometimes called simply ``the accessible set from $x$'').


\begin{definition}[Global controllability]\label{def:controllabilty-glob} 
\hspace{.3em}\textup{\textup{I.}}\hspace{.1em}
 System \eqref{sys} is \textbf{globally controllable} if and only if
    $\Acc^{U}(x)=\bigmanif$ for all $x$ in $\bigmanif$.
\hspace{.4em}
\textup{\textup{II.}}\hspace{.1em}
Let $\pi : \bigmanif \!\to\! \manifM$ be some fibration\footnote{
  This $\pi$ need not be related to the construction from \cref{sec-sys}
  ($\vecfield^0$ is even not assumed to satisfy \eqref{per});
  being a fibration could be replaced by the weaker requirement of being a submersion.
}, 
with $\manifM$ a manifold of arbitrary dimension $n<d$.
System \eqref{sys} is \textbf{\boldmath globally controllable with respect to $\pi$} if and only if
$\pi\!\left(\Acc^{U}(x)\right)=\manifM$ for all $x$ in $\bigmanif$.
\end{definition}
\begin{remark}
  \label{rem-2}
  Global controllability (\textup{\textup{I}}) obviously implies global
  controllability with respect to $\pi$ (\textup{\textup{II}}). The converse does not hold: take any non
  controllable system and $\manifM\!=\!\{0\}$, $n=0$, $\pi$ constant.
  However, under the periodicity Assumption \eqref{per} and if $\pi$ is the fibration 
  induced by the periodic orbits (\textit{cf.} \cref{prop-pi}), global
  controllability with respect to $\pi$ does imply global
  controllability according to \cref{rem-ext}. 
\end{remark}

Controllability \emph{with respect to $\pi$} could also be called \emph{partial controllability},
a familiar notion where only the partial information $\pi(x)$ on the
final point (at final time) is
specified rather than $x$ itself.
In \cite{Danh-Loh-Jun21,Krei-Sar64,Sara-Kra63}
it is called \emph{output controllability}, as
$\pi$ can be viewed as an $n$-dimensional output of the system, specified at final time.
  In \cite{Lewi-Mur99} or \cite[Section 7.2]{Bullo-Lewis}, the notion of
\emph{configuration controllability} for mechanical systems is similar:
roughly speaking, $\pi\!:TQ\to Q$ is defined by
$\pi(q,\dot q)=q$, where $q$ (position) is the current point in the configuration
manifold $\manifM=Q$ and $(q,\dot q)$ (position-velocity) the one in the
state manifold $\bigmanif=TQ$, tangent bundle of $Q$;
there is however a slight difference in the choice of initial conditions
(zero velocity initial conditions are privileged
in~\cite{Bullo-Lewis,Lewi-Mur99}).

Let us now pursue with \emph{local} controllability,
either around a particular solution, with a specified final time, 
or around a particular \emph{periodic} solution, with unprescribed final time.
In both cases, the particular solution is one of $\dot x=\vecfield^0(x)$.
The given definitions would be similar for a solution $t\mapsto(\bar x(t),\bar u(t))$ of \eqref{sys} with a nonzero control, and the results of the paper could be extended, but it does not meet the original motivations of the paper.
%
  The following notation is introduced to make \cref{def:controllabilty-loc} concise. Let
\begin{equation}
  \label{eq:9}
  \begin{split}
    \BAcc_{\tf}^{U,\Omega}(x_0)=\{ x(\tf),\;
  &\text{with }t\mapsto(x(t),u(t))
  \text{ solution of \eqref{sys} on }[0,\tf]\\[-1ex]
  &\quad\text{ such that }
  x(0)=x_0 \text{ and }x(t)\in\Omega
  \text{ for all $t$ in }[0,\tf]\}\,,\hspace{-.5em}
  \end{split}
\end{equation}
with $U,x_0,\tf$ as above and $\Omega$ some open subset of $\bigmanif$.
Compared to the accessible sets $\!\Acc$ defined above, we added
the requirement that the trajectories $x(.)$ have to stay in $\Omega$. Thus,
$\Acc_{\tf}^{U}(x_0)=\BAcc_{\tf}^{U,\bigmanif}(x_0)$.

\begin{definition}[Local controllability]
  \label{def:controllabilty-loc}
  Let $t\mapsto\bar x(t)$  be a particular solution of $\dot x=\vecfield^0(x)$,
  defined on $[0,\tf]$ for some $\tf>0$, and let $\Gamma=\bar x([0,\tf])$.
\\[.6ex]
\textup{I.}\hspace{.6em}
System~\eqref{sys} is \textbf{\boldmath locally controllable around the solution $\bar x(.)$}
   if and only if \\the set
$\BAcc_{\tf}^{U,\Omega}(\bar x(0))$
is a neighborhood of $\bar x(\tf)$ for any neighborhood $\Omega$ of $\Gamma$ 
in $\bigmanif$.
\\[.6ex]
\textup{II.}\hspace{.6em}
Let $\pi\!: \bigmanif \to \manifM$ be as in \textup{\cref{def:controllabilty-glob}-II}.
  System~\eqref{sys} is \textbf{\boldmath locally controllable with respect to $\pi$ around the solution}
  $\bar x(.)$ if and only if the set
$$\pi\left(\BAcc_{\tf}^{U,\pi^{-1}(\Omega')}(\bar x(0))\right)$$
is a neighborhood of
$\pi\left(\bar x(\tf)\right)$ in $\manifM$ for any neighborhood
$\Omega'$ of $\pi(\Gamma)$ in $\manifM$.
\\[.6ex]
\textup{III.}\hspace{.6em}
Assume that the drift vector field $\vecfield^0(x)$ satisfies \eqref{per}, let $\pi\!: \bigmanif \to \manifM$
be the fibration with base $M$ equal to the respective orbit space, and let
$\bar x(.)$ be one of these periodic orbits, 
$\tf>T(\bar x(0))$.
\\
System~\eqref{sys-Iphi} is
  \textbf{locally orbitally controllable around the periodic solution $\bar x(.)$}\\
  if and only if the set
  $$\bigcup_{t\in[0,+\infty)}\pi\Bigl(\BAcc_{t}^{U,\pi^{-1}(\Omega')}(\bar x(0))\Bigr)$$
is a neighborhood of $\pi(\bar x(0))$ in $\manifM$ for any neighborhood
$\Omega'$ of $\pi(\bar x(0))$ in $\manifM$.
\end{definition}
In part III, since \eqref{per} holds, the definition may be
re-formulated using the product form
\eqref{sys-Iphi} locally on some open $O\subset \manifM$ containing
$\bar I=\pi(\bar x(0))$ (in which the solution $t\mapsto\bar x(t)$
reads $t\mapsto(\bar I,\bar\varphi+\omega(\bar I\,)\,t)$):
System~\eqref{sys-Iphi}, or \eqref{sys}, is
\emph{locally controllable around the periodic solution $\bar I$}
(it does not depend on $\bar\varphi$)
  if and only if the set
\begin{equation}
  \label{eq:401}
  \begin{split}
    \hspace{-.2em}
    \{ I(\tf),\ \tf\!\in\![0,+\infty)
    \text{ and }
    t\mapsto(I(t),\varphi(t),u(t))
  \text{ solution of \eqref{sys-Iphi} on }[0,\tf]\text{ such }\hspace{.8em}\\
  \text{that }
  I(0)=\bar I \text{ and }
  I(t)\in\Omega' \text{ for all $t$ in }[0,\tf]\}
  \end{split}
\end{equation}
is a neighborhood of $\bar I$ in $\manifM$ for any neighborhood
$\Omega'$ of $\bar I$ in $\manifM$.
\begin{remark}
  \label{rem-condloc}
  We give necessary conditions for these properties in
  \cref{sec-obstruc}; while
  \cref{thm:xbar,thm:1tour} in \cref{sec-xbar,sec-conseq} give sufficient conditions,
  when $U$ is convex and compact, for stronger
  local controllability properties. These properties require, for instance, instead
  of  local controllability around $\bar x(.)$ as in Definition~\ref{def:controllabilty-loc}-I above, that
``for any $\varepsilon>0$,  $\Acc^{\varepsilon U}_{\tf}\!(\bar x(0))$ is a neighborhood of $\bar x(\tf)$''.
  This implies the above ``local controllability around $\bar x(.)$'' because,
for any neighborhood $\Omega$ of $\Gamma$, one has
$\Acc_{\tf}^{\varepsilon U}(x_0)\subset\BAcc_{\tf}^{U,\Omega}(x_0)$
for $\varepsilon>0$ small enough, from continuous
dependance of the solutions on the control, and the fact that $\varepsilon U\subset U$ if
$\varepsilon\in[0,1]$.
\end{remark}

\begin{remark}
  \label{rem-loc}
  While \cref{def:controllabilty-glob} of global controllability is classical,
  local controllability is a moving notion in the literature, so that
\cref{def:controllabilty-loc}  
  agrees with only part of the literature.
  Two remarks are in order.
\\-
  In our setting, local controllability  is at the same time weaker
  (it only deals with reaching some neighborhood of the end point of a trajectory)
  and stronger (trajectories are required to remain close to the reference one)
  than global controllability.
  This is illustrated at the end of \cref{exemple-illustration}.
\\-
  We believe that a ``local'' property of a control system at some locus should only depend
  on the description of this system on arbitrary small neighborhoods of that locus.  In this
  sense, \cref{def:controllabilty-loc} is
  local with respect to the state only:
  System~\eqref{sys} is locally controllable around the solution $\bar x(.)$
  if and only if any system $\dot x=f(x,u)$ such that
  $f(x,u) = X^0(x)\!+\!\sum u_k\vecfield^k(x)$ for $(x,u)$ in $\mathcal{O}\times U$,
  with $\mathcal{O}$ an open neighborhood of $\bar x([0,\tf])$,
  is locally controllable around the solution $\bar x(.)$.
\end{remark}



\subsection{Some families of vector fields}
\label{sec:LieBra}
Let $V(\bigmanif)$ be the set of smooth (\emph{i.e.} either $C^\infty$ or $C^\omega$) vector fields on $\bigmanif$;
the Lie bracket of two elements of $V(\bigmanif)$ (defined in any advanced calculus 
textbook) is also in
$V(\bigmanif)$, making $V(\bigmanif)$ a Lie algebra over the field $\RR$.
A family of vector fields is a subset $\mathcal{F}\subset V(\bigmanif)$;
we denote by $\mathcal{F}(x)$ the subset of $T_x\bigmanif$ made
of the values at $x$ of vector fields in $\mathcal{F}$:
    \begin{equation}
      \label{eq:8}
      \mathcal{F}(x)=\{X(x),\; X\in\mathcal{F}\}\subset T_x\bigmanif,
    \end{equation}
and by $\vect\mathcal{F}(x)$ the vector subspace of $T_x\bigmanif$ spanned by
$\mathcal{F}(x)$.
Finally, we denote  by $\AccFam_{\mathcal{F}}(x)$ the \emph{accessible set from $x$ of this
  family, i.e.}\ the
set of points that can be
reached from $x$ by following successively the flow of a finite number of
vector fields in $\mathcal{F}$, each for a certain positive time; this notion differs from $\Acc^{U}\!(x)$ in \eqref{eq:acc},
 attached to the control system \eqref{sys}, that allows
  arbitrary $\mathrm{L}^\infty_{\mathrm{loc}}$ controls.

Let us define two families of vector fields from $\vecfield^0, \ldots,\vecfield^m$ defining System~\eqref{sys}:
\begin{align}
  \label{eq:45}
    \mathcal{L}=&\;\Bigl\{
                  [\vecfield^{i_{k}},[\vecfield^{i_{k-1}},[\cdots\cdots[\vecfield^{i_{2}},\vecfield^{i_{1}}]\cdots]]]
                  ,
                  \\[-1.8ex]\nonumber&\hspace{9em}
                  (i_{1},\ldots,i_{k})\in\{0,\ldots,m\}^k
                  ,\,k\in\NN\setminus\{0\}
             \Bigr\}
  \,,\hspace{-2em}
  \\[-1ex]
      \label{eq:3}
  \mathcal{F}_0=&\;\bigl\{\mathrm{ad}_{\vecfield^0}^j \vecfield^k,\
                 k\in\{1,\ldots,m\},\ j\in\NN\bigr\} \,.
\end{align}
Recall that $\mathrm{ad}_X^j Y$ is defined by $\mathrm{ad}_X^0 Y = Y,$ and $\mathrm{ad}_X^j Y = [X, \mathrm{ad}_X^{j-1} Y]$. $\mathcal{L}$ is used in \cref{section:state-of-the-art} to
	discuss literature, and
$\mathcal{F}_0$ in \cref{sec-comments+examples} to discuss our results.
\begin{definition}[Bracket generating property]\label{def:larc}
  The control system \eqref{sys}, or
  the family of vector fields $\{\vecfield^0,\ldots,\vecfield^m\}$, is
  called \textit{bracket generating at point} $x$ 
  if and only if $\, \vect\mathcal{L}(x)=T_x\bigmanif$, and \emph{bracket generating} if
  this is true for all $x$ in $\bigmanif$.
\end{definition}

\noindent
Being bracket generating is also called the Lie Algebra
Rank Condition (LARC), see \cite{Suss73}.
  In most of the literature
  (\emph{e.g.}\ \cite[chapter 8]{agrachev_control_2004} or 
  \cite[chapter 3]{jurdjevic_geometric_1996}),
  one proves controllability properties by piecewise constant controls, \emph{i.e.}  analyzes $\AccFam_{\mathcal{G}}(x)$ for the family of vector fields
\begin{equation}
  \label{eq:cdc1}
  \mathcal{G}=\{\,\vecfield^0+u_1\vecfield^1+\cdots+u_m\vecfield^m\,,\;(u_1,\ldots,u_m)\in U\}.
\end{equation}
See the lines after \eqref{eq:8} for a definition of
$\AccFam_{\mathcal{G}}(x)$, which is also the set of points that can be
reached, for the Control system \eqref{sys}, with piecewise constant $U$-valued
controls, to the effect that $\AccFam_{\mathcal{G}}(x)\subset\Acc^U(x)$.
If $U$ satisfies \eqref{eq:assU},
then $\vect\{\vecfield^0,\ldots,\vecfield^m\}(x)=\vect\mathcal{G}(x)$ for all $x$, so that
$\mathcal{G}$ is bracket generating if and only
if $\{\vecfield^0,\ldots,\vecfield^m\}$ is.

\subsection{\boldmath Controllability when $\co U$ is a neighborhood of the origin, state of the art}
\label{section:state-of-the-art}
  It is natural to assume that
  the family $\{\vecfield^0,\ldots,\vecfield^m\}$ is bracket generating because this
  property is
(for example according to the Sussman's Orbit Theorem~\cite{Suss73})
  necessary for controllability if the vector fields are real analytic,
  which is  a reasonable framework although the results of the present paper do not require it.
Bracket generating is not however sufficient,
  it rather implies a weaker property, where the accessible set from a point has a
  nonempty topological interior,
  see the work of Sussmann and Jurdjevic \cite{Suss-Jur72} or Krener \cite{Kren74}.

One case where the bracket generating property, assuming that $U$ is a neighborhood of the origin, 
implies controllability is the one of driftless systems, namely Systems
\eqref{sys} with $\vecfield^0=0$, see, \textit{e.g.}, the work of Lobry~\cite{Lobr70}.
This is however far from Assumption~\eqref{per}, where the drift is assumed to be non-vanishing.
It turns out that controllability is also obtained under this periodicity assumption, or the more general property of
Poisson stability.
Recall that, for a complete vector field $X^0$ on $\bigmanif$,
a point $x\in\bigmanif$ is said to be
\emph{Poisson stable} for $X^0$ if there exists a sequence of positive times
  $t_n \to \infty$ such that $\exp(t_n X^0)(x) \to x$ when $n \to \infty$;
  $X^0$ is said to be Poisson stable if there is a dense subset of such
  points.
It turns out that many physical dynamical systems have this property; this makes the following result quite
useful.
\begin{theorem}[Bonnard, 1981, \cite{Bonn81cras}]
  \label{th:BB}
   System \eqref{sys} is globally controllable if
  \\ \hspace*{2em}\emph{(i)} the vector field $X^0$ is Poisson stable, 
  \\ \hspace*{2em}\emph{(ii)} the family $\{X^0,X^1,\ldots,X^m\}$ is bracket generating, and 
  \\ \hspace*{2em}\emph{(iii)} $\co U$
                       is a neighborhood of $0$ in $\RR^m$.
\end{theorem}
\noindent It is stated in this form in the textbook
\cite{jurdjevic_geometric_1996} (Chapter 4, Theorem 5) or in the original reference, that mentions techniques due to \cite{Jurd-Kup81}.
It has been widely used,
for instance to prove controllability prior to solving an optimal control problem, like in
\cite{Cail-Noa01,lian-1994a} for the controlled Kepler problem.
Assumption (iii) is essential to \cref{th:BB},
as evidenced by a very simple academic example of the form \eqref{sys}
with $\bigmanif=\RR\times\Sun$
and one scalar control $u$ ($m=1$):
$\,\dot I=u\,,\ \dot\varphi=1$, $u\in U=[0,1]$;
conditions (i) and (ii) are satisfied, (iii) is not, and 
controllability cannot hold because $I\in\RR$ cannot decrease.
  Our \cref{thm:CDC-gen}
  can be seen as a generalization of \cref{th:BB} to the case where (iii) does not hold.
  There we keep (i) (and even the stronger Condition \eqref{per})
  but we have to replace (ii) by a stronger condition,
  involving not only the vector fields but also the set $U$.
  Next section introduces the necessary tools.

\subsection{Further constructions, not based on Lie brackets}
\label{sec-notreE}

To the vector fields $\vecfield^0,\ldots,\vecfield^m$ and the
set $U$ defining System \eqref{sys}, we associate, for any $x$ in $\bigmanif$ and any
real numbers $\tau$, or $t_1<t_2$,
the following subsets\footnote{
  These subsets are related to control variations in the proof of
  the Pontryagin Maximum Principle.  For instance,
  the condition of \cite[Lemma 12.6]{agrachev_control_2004} is in
  terms of $E^U_{[t_1,t_2]}(x)$
  (with $\tilde u\equiv0$, $\mathcal{T}=[t_1,t_2]$,
  $g_{\tau,u}(x)=\sum_{k=1}^mu_k\left(\exp(-\tau\vecfield^0)_*\vecfield^k\right)\!(x)$).
  Our \cref{thm:xbar} can be viewed as a consequence of that lemma in
  the simpler case where the reference control is smooth.
}
of $T_x\bigmanif$, where the pushforward notation ``${}_*$''
  was recalled in \eqref{pushforward}:
\begin{align}
  \label{eq:1}
  E^U_\tau\!(x)&=\left\{
             \sum_{k=1}^m u_k
             \left(\exp(-\tau\vecfield^0)_*\vecfield^k\right)\!(x),\;
             u\in U \right\}
                 \subset T_x\bigmanif\,,
\\
  \label{eq:EUt1t2}
                 E^U_{[t_1,t_2]}\!(x)
  \,
               &
                 =\!\bigcup_{\tau\in[t_1,t_2]}\!\!\!\!\!
                 E^U_\tau\!(x)\subset T_x\bigmanif\,,
\\[-1.4ex]
  \intertext{and, when the periodicity assumption \eqref{per} holds and defines $T(x)$,}
  \label{eq:2}
       E^U\!(x)&=
                 E^U_{[0,T(x)]}\!(x)\,=\!\bigcup_{\tau\in[0,T(x)]}\!\!\!\!\!
                 E^U_\tau\!(x)\subset T_x\bigmanif\,.
                 \\[-5ex]\nonumber
\end{align}
Note that $E^U(.)$ and the function $T(.)$, are invariant under the flow of $\vecfield^0$.

\begin{remark}
    In the particular case when the fibration $\pi$ is the one associated with the orbits
    of the periodic drift, using the product presentation $(I,\varphi)$ from
    Equations~\eqref{Xi-Fi}--\eqref{sys-Iphi}, with
    $\varphi\in\Sun$ and $I\in O\subset\manifM$, one may
    simplify the expressions in
\eqref{eq:1}, \eqref{eq:EUt1t2} and \eqref{eq:2} because
the construction leads to
$\vecfield^0=\omega\,\partial/\partial\varphi$, hence the explicit expression of
the flow of $\vecfield^0$:
$\exp(t\vecfield^0)(I,\varphi)=(I,\varphi+\omega(I)\,t)$, so that
$\left(\exp(-\tau\vecfield^0)_*Z\right)(I,\varphi)=Z(I,\varphi+\omega(I)\,\tau)
-\tau\,Z\omega (I,\varphi+\omega(I)\tau) \,\partial/\partial \varphi$
  for any vector field $Z$ ($Z\omega$ stands for the Lie derivative along
  $Z$ of $\omega:O\times\Sun\to\RR$).
  In particular, with $x=(I,\varphi)$, denoting by $\pi'(x)$ the derivative (tangent map) 
  of the fibration $\pi$ at point $x$ (\textit{i.e.} the linear projection on the first
  factor in $T_x\bigmanif=T_I \manifM\times\RR$),
  the projection of $E^U(x)$ has a simple expression:
\begin{equation}
  \label{eq:26a}
  \begin{split}
  \pi'(I,\varphi) \bigl(E^U\!(I,\varphi)\bigr)=\breve{E}^U(I)\subset T_I\manifM\,& \quad
  \text{with}
  \\[-.4ex]
  \breve{E}^U\!(I)
  =\left\{\sum_{k=1}^mu_k\,F^k(I,\varphi)\,,
    \right.&\left.\vphantom{\sum_{k=1}^m}
      \ \;u\in U\,,\;\varphi\in\Sun\right\} \,.
    \\[-1.6ex]
\end{split}
\end{equation}
Here,
invariance of $E^U\!(.)$ by the flow of $\vecfield^0$ implies that $\breve{E}^U$ does not
depend on $\varphi$.
\end{remark}

\section{Global and local controllability results}
\label{sec-jurdj}

All results are concerned with System~\eqref{sys}
(or with System \eqref{sys-Iphi} in product form if $\vecfield^0$ satisfies \eqref{per}).
A few remarks are in order on the assumptions in these results.

\smallskip

\paragraph{Smoothness} 
All vector fields are assumed to be smooth, \textit{i.e.}\ infinitely differentiable, 
although the reader may check that all proofs work if smooth is replaced by
differentiable with a locally Lipschitz continuous first derivative.

\smallskip

\paragraph{Periodicity of solutions of the drift}
This Property \eqref{per} of $\vecfield^0$, discussed in \cref{sec-sys},
is assumed in \cref{thm:CDC-gen}, that was a primary motivation of this work, but is \emph{not}
a running assumption.
Also, we indicate when $\pi:\bigmanif\to\manifM$ is the fibration induced by \eqref{per},
with base $M$ equal to the respective orbit space; it is otherwise an arbitrary
fibration, unrelated to $\vecfield^0$, with $M$ of arbitrary dimension $n\leq d$.

\smallskip

\paragraph{The control set $U$} It is always assumed to satisfy \eqref{eq:assU}.
We also need $U$ to be convex and compact in \cref{sec-xbar,sec-conseq}
(\cref{thm:xbar,thm:1tour}, see \cref{rem-U}).



\medskip

\paragraph{Notation} For a subset $S$ of a real vector space $E$, its \emph{conic hull}
is the smallest convex cone containing $S$, i.e. the set of \emph{conic combinations} of elements of
$S$:
\begin{equation}
  \label{eq:defcone}
  \cone S=\bigl\{\,\lambda_1\,x_1+\cdots+\lambda_k\,x_k,\;
x_i\in S,\, \lambda_i\in[0,+\infty),\,i\in\{1\ldots k\},\,k\in\NN\,\bigr\}
\,.
\end{equation}

\subsection{Global controllability under the periodicity assumption on the drift}
\label{sec-jurdj-results}

Let us now state our generalization of the now classical
\cref{th:BB} on global controllability.
Here, the fibration $\pi$ is the one induced by the periodic
drift, see \cref{prop-pi}.
Global controllability is as in \cref{def:controllabilty-glob};
as noticed in \cref{rem-2}, it is equivalent, in the present case, to ``global
controllability with respect to $\pi$''.

\begin{theorem}
  \label{thm:CDC-gen}
  Assume that $U$ satisfies \eqref{eq:assU},
  that $\vecfield^0$ satisfies \eqref{per} (periodicity of all solutions),
  inducing a fibration $\pi$ as in \cref{prop-pi},
  and that
    \begin{equation}
      \label{eq:10}
      \pi'(x)\bigl(\cone E^U\!(x)\bigr) =  T_{\pi(x)}\manifM
      \ \ \text{for all $x$ in $\bigmanif$.}
     \end{equation}
Then System~\eqref{sys} is globally controllable,
and \emph{a fortiori} globally controllable with respect to $\pi$.
Under the same conditions, global controllability holds as well in backward time, \emph{i.e.}
$\Acc^U_{\mathbf{-}}(x)=\bigmanif$ for all $x$ in $\bigmanif$.
\end{theorem}
\noindent
  Using the local product form
\eqref{sys-Iphi} and the notation $\breve{E}$ from \eqref{eq:26a},
Condition~\eqref{eq:10} for all $x$ in $\bigmanif$ equates to
the following relation for all $I$:
\begin{equation}
  \label{eq:10bis}
  \cone  \breve{E}^U\!(I) = T_I\manifM
  \ \ \text{for all $I$ in $\manifM$.}
\end{equation}

\begin{proof}
  We prove controllability in forward time; backward time follows upon changing each
  vector field $\vecfield^i$ into $-\vecfield^i$, preserving all assumptions of the theorem.
The accessible set $\AccFam_{\mathcal{F}}(x)$ of
    a family $\mathcal{F}$ of vector fields was defined in the
    first paragraph of \cref{sec:LieBra}. 
    The family $\mathcal{G}$ was defined in \eqref{eq:cdc1}, and we
    noted there that it satisfies
    $\AccFam_{\mathcal{G}}(x)\subset\Acc^U\!(x)$ for all $x$.
    The conclusion of the theorem is then obviously implied by
    $\AccFam_{\mathcal{G}}(x)=\bigmanif$ for all $x$ in $\bigmanif$,
    that we are going to prove.

Define, from $\mathcal{G}$, the families $\mathcal{G}_1$ and $\mathcal{G}_2$ by
\begin{equation}
  \label{eq:cdc2}
  \mathcal{G}_1= \mathcal{G}\cup \{-\vecfield^0\}
  \quad \text{and}\quad
  \mathcal{G}_2=\{\,\exp(t\,\vecfield^0)_* X\,,\;X\in\mathcal{G}_1,\,t\in\RR\}\,, 
\end{equation}
where the push-forward notation $\exp(t\,\vecfield^0)_* X$ was recalled in \eqref{pushforward}.
Obviously, $\mathcal{G}\subset\mathcal{G}_1\subset\mathcal{G}_2$.
Let us first prove that, for all $x$, 
\begin{equation}
  \label{eq:cdc-7}
  \AccFam_{\mathcal{G}}(x)=
  \AccFam_{\mathcal{G}_1}(x)=
  \AccFam_{\mathcal{G}_2}(x)\,.
\end{equation}
The family $\mathcal{G}$ contains $\vecfield^0$ because $0\in U$, see \eqref{eq:assU};
hence $\exp(t\vecfield^0)(x)\!\in\!\AccFam_{\mathcal{G}}(x)$ for any $x$
in $\bigmanif$ and $t\!\geq\!0$; it is also true if $t\!<\!0$
because \eqref{per} implies
$\exp(t\vecfield^0)(x)=\exp\bigl((t+kT)\vecfield^0\bigr)(x)$,
where the integer $k$ can be taken large enough that
$t+kT(x)\!\geq\!0$;
the reverse inclusion being clear, this proves the first equality.
Turning to the second one, consider some $y$ in $\AccFam_{\mathcal{G}_2}(x)$. By definition,
$y=\exp(s_1Y^1)\circ\cdots\circ\exp(s_{\!N}Y^{\!N})(x)$ for some
positive integer $N$, and some positive
real numbers $s_j$ and vector fields $Y^j$, $j\in\{1,\ldots,N\}$, belonging to $\mathcal{G}_2$, such that, for some real number $t_j$ and 
some $Z^j\in\mathcal{G}_1$, $Y^j=\exp(t_j\vecfield^0)_*Z^j$. Since the
latter implies
$\exp(s_j\,Y^j)=\exp(t_j\vecfield^0)\circ\exp(s_1\,Z^1)\circ\exp(-t_j\vecfield^0)$,
and both $\vecfield^0$ and $-\vecfield^0$ belong to $\mathcal{G}_1$,
the above yields $y$ as the image of $x$ by a composition of $3N$ flows in
positive time of vector fields belonging to $\mathcal{G}_1$. Hence
$y\in\AccFam_{\mathcal{G}_1}(x)$, which proves the second equality in
\eqref{eq:cdc-7}, the reverse inclusion being, again, obvious. This
ends the proof of \eqref{eq:cdc-7}.

For any $x$ in $\bigmanif$, \eqref{eq:10} implies $\cone\mathcal{G}_2(x)\!=\!T_x\bigmanif$,
hence there are $d+1$ vector fields $Y^0,\ldots,Y^d$ in $\mathcal{G}_2$ such that
$Y^0(x),\ldots,Y^d(x)$ generate $T_x\bigmanif$ by conic combinations.
Let $\phi\!:(-\alpha,\alpha)^d\to\bigmanif$ (with $\alpha\!>\!0$
small enough so that the flows are defined) be the
smooth map given by
$\phi(s_0,\ldots,s_d) = \exp(s_0Y^0)\circ\cdots\circ\exp(s_dY^d)(x)$,
and apply \cref{lem-open-map} (see below)
with $k\!=\!d+1$, $W^k\!=\!(-\alpha,\alpha)^{d+1}$, $N\!=\!\bigmanif$
and $G\!=\!\phi$.
Since $\phi'(0)\bigl(\RR^k_+\bigr)=T_{\phi(0)}\bigmanif$ (this follows
from the fact that $\partial\phi/\partial s_i(0)\!=\! Y^i(x)$, $i\in\{0\ldots d\}$, and the conic
combination property of $Y^0 \ldots Y^d$ stated above),
the lemma tells us 
that $\phi(W^k_+)$ is a neighborhood of $x$.
Since $\phi(W^k_+)=\phi([0,\alpha)^{d+1})$ is contained in
$\AccFam_{\mathcal{G}_2}(x)=\AccFam_{\mathcal{G}}(x)$ (see \eqref{eq:cdc2}), we
have shown the following property, valid also for $\AccFam_{-\mathcal{G}}(x)$ by changing each
$\vecfield^i$ into $-\vecfield^i$:
\begin{equation}
  \label{eq:114}
  \text{for all $x$ in $\bigmanif$, }
  \AccFam_{\mathcal{G}}(x) \text{ and } \AccFam_{-\mathcal{G}}(x) \text{ are neighborhoods of $x$ in $\bigmanif$.}
\end{equation}
Any $\AccFam_{\mathcal{G}}(x)$ is open because, 
for $y$ in $\AccFam_{\mathcal{G}}(x)$,
applying \eqref{eq:114} to $\AccFam_{\mathcal{G}}(y)$,
the latter is a neighborhood of $y$, contained in $\AccFam_{\mathcal{G}}(x)$ by
definition. 
Any $\AccFam_{\mathcal{G}}(x)$ is also closed: picking  
$y$ in its topological closure,
$\AccFam_{-\mathcal{G}}(y)$ is a neighborhood of $y$ according to
\eqref{eq:114}, so $\AccFam_{-\mathcal{G}}(y)$ intersects $\AccFam_{\mathcal{G}}(x)$ and 
$y$ must belong to $\AccFam_{\mathcal{G}}(x)$. Being an open, closed, nonempty subset of the 
connected manifold $\bigmanif$, $\AccFam_{-\mathcal{G}}(y)$ is equal to $\bigmanif$.
\end{proof}

The following lemma is used in the above proof, and in the proof
of \cref{thm:xbar}.
\begin{lemma} [Open mapping lemma, see
  \mbox{\cite[Theorem~2.1.3]{Sussmann1998}},
  \mbox{\cite[Lemma~12.4]{agrachev_control_2004}}, or
  \mbox{\cite[Theorem 5]{Grav50}}]
  \label{lem-open-map}
  With $k$ a positive integer, let $\RR^k_+$ be the positive orthant in $\RR^k$:
\begin{equation}
  \label{eq:orthant}
  \RR^k_+=\{(w_1,\ldots,w_k)\!\in\!\RR^k\!,\,w_i\!\geq\!0,\,i\!=\!1\ldots k\}\,,
\end{equation}
  $W^k$ an open neighborhood of $0\!\in\!\RR^k$, and
  $W^k_+\!=\!W^k\cap\RR^k_+$.
    If $G\!:W^k\to N$ is a smooth map, with $N$ a smooth manifold, such that
    $G'(0)\!\cdot\! \RR^k_+=T_{G(0)}N$, then $G(V \cap W^k_+)$ is a neighborhood of
    $G(0)$ in $N$ for any neighborhood $V$ of the origin in $\RR^k$.
\end{lemma}
We state the smooth case because it is what we need. The
proof in \cite{agrachev_control_2004,Sussmann1998} uses Brouwer's
fixed point theorem under the much
weaker assumption that $G$ is continuous and
differentiable at zero (Lipschitz continuous in \cite{agrachev_control_2004}).
In \cite{Grav50}, $\RR^k$ and $N$ are replaced by Banach spaces under a condition slightly
weaker than continuous differentiability; that result also implies \cref{lem-open-map}.

\subsection{Local controllability around a particular solution}
\label{sec-xbar}
Let us consider 
\begin{equation}
  \label{eq:014}
t\mapsto\bar x(t)\,,\ \text{solution of}\ \dot x=\vecfield^0(x), \
\text{defined on the time-interval }[0,\tf]
\end{equation}
(the real number $\tf$ may be negative, $[0,\tf]$ should then be
understood as $[\tf,0]$),
and define the following subset of $T_{\bar x(0)}\bigmanif$, attached to the solution $\bar x(.)$:
\begin{equation}
  \label{eq:19}
  \bar E=E^U_{[0,\tf]}(\bar{x}(0))\,,
\end{equation}
with $E^U_{[0,\tf]}\!(x)$ defined in \eqref{eq:EUt1t2}.
Next theorem gives, in terms of $\bar E$, sufficient conditions for
local controllability along $\bar x(.)$.

Local controllability (with respect to $\pi$ or not) is introduced in
\cref{def:controllabilty-loc}, and \cref{rem-condloc} explains why Conditions \eqref{eq:012} and
  \eqref{eq:013} are stronger.
  Part~I of \cref{def:controllabilty-loc} is a particular case of Part~II, with $\manifM\!=\!\bigmanif$ and
  $\pi\!=\!\mathrm{Id}$; it is stated independently for readability
  ($\pi\!=\!\mathrm{Id}$ yields the simpler sufficient Condition \eqref{eq:13}).
\begin{theorem}
  \label{thm:xbar}
  Consider System \eqref{sys} with $U$ a convex compact 
  subset of $\RR^m$ containing the origin, and let $\bar x(.)$ and $\bar E$ be as
  described above.

\noindent
  \textup{I.} If
  \begin{equation}
  \label{eq:13}
  \cone \bar E = T_{\bar x(0)}\bigmanif\,,
\end{equation}
then System \eqref{sys} is locally controllable around $\bar x(.)$ and, more precisely,
\begin{equation}
 \label{eq:012}
  \text{$\Acc^{\varepsilon U}_{\tf}\!(\bar x(0))$ is a neighborhood of $\bar x(\tf)$ in $\bigmanif$ for any $\varepsilon>0$.}
\end{equation}

\noindent
\textup{II.}  Let $\pi\!:\bigmanif\to\manifM$ be a smooth fibration
  with $\manifM$ of dimension $n\leq d$.
  If
\begin{equation}
  \label{eq:13bb}
    \left( \pi\circ\exp(\tf\vecfield^0) \right)'\!\!(\bar x(0))\,
  (\cone \bar E) = 
T_{\pi(\bar x(\tf))}\manifM\,,
\end{equation}
then System \eqref{sys} is locally controllable around $\bar x(.)$ with respect to $\pi$; more precisely,
\begin{equation}
 \label{eq:013}
\pi\!\bigl(\Acc^{\varepsilon U}_{\tf}\!(\bar x(0))\bigr)
\text{is a neighborhood of $\pi(\bar x(\tf))$ in $\manifM$ for any $\varepsilon>0$.}
\end{equation}
\end{theorem}

\begin{remark} 
  \label{rmk-Ebar-depend-de-tf}
  Since  $\bar{x}(.)$ and $\tf$ are
  fixed in the theorem, we use the short notation $\bar E$; this 
    should not hide the dependence of this set on the 
    initial condition and more interestingly on $\tf$.
    It is clear from \eqref{eq:EUt1t2} that $E^U_{[0,t]}(x)$ grows with $t>0$ for fixed $x$;
    hence, if one follows the same solution (assumed to be defined on $[0,+\infty)$), either
    Condition \eqref{eq:13} is satisfied for all $\tf$
    or it is not satisfied for small values of $\tf$, but is satisfied, if ever, for $\tf$ larger than a
    certain value.
  Note that it is not necessarily true for \eqref{eq:13bb}, unless $\pi$ is invariant by
  the flow of $\vecfield^0$ ($\pi\circ\exp(t\vecfield^0)=\pi$), which is indeed the case
  when \eqref{per} holds and $\pi$ is associated to periodic orbits of $\vecfield^0$.
\end{remark}
\begin{remark}[assumptions on $U$]
  \label{rem-U}
    \Cref{thm:xbar,thm:1tour}, unlike \cref{thm:CDC-gen}, 
    assume $U$ convex and compact.
    Compactness is natural; it a priori avoids unbounded controls and
    allows $\varepsilon U$ to be small for small $\varepsilon>0$.
    Convexity is more restrictive, it is technically needed in the proofs, and could
    probably be relaxed. 
    If $U$ is not convex, the same conclusions hold, replacing
    $\Acc^{\varepsilon U}_{\tf}\!(\bar x(0))$ with its topological closure.
\end{remark}

Before proving \cref{thm:xbar}, let us recall that the linear approximation of the
control system \eqref{sys} along the solution $\bar x(.)$ is the following linear
time-varying system:
\begin{equation}
    \label{eq:24}
    \begin{split}
      \dot{\xi} = \,&\sum_{k=1}^m\deltau_k  \left(\exp(-t\vecfield^0)_*\vecfield^k\right)\!(\bar x(0)) \,,
      \\
      \deltax(t)=\,&\exp(t\vecfield^0)'(\bar x(0))\, \xi(t)\,,
    \end{split}
\end{equation}
with initial condition $\xi(0)=0$, where 
$\deltau=(\deltau_1,\ldots,\deltau_m)$ is a small variation of the control
around zero and $\deltax(t)$ the corresponding small variation of
the state around $\bar{x}(t)$
(in other words, $t\mapsto(\bar x(t)\!+\!\deltax(t),\deltau(t))$ is a
solution of \eqref{sys}).
Note that \eqref{eq:24} is the intrinsic way of writing
$\dot{\deltax}
  =
  \frac{\partial \vecfield^0}{\partial x}(\bar x(t))\,\deltax
  +\sum_1^m\deltau_k \, \vecfield^k(\bar x(t))$,
because $\xi$ is an element of the tangent space $T_{\bar x(0)}$ for all $t$, contrary to $\deltax$.
One may constrain the control $\deltau=(\deltau_1,\ldots,\deltau_m)$
to some
$U\subset\RR^m$
and define,
for this linear control-constrained time-varying system,
the accessible set
  for initial time zero, initial state $x_0$ and final time $t$ as
the following subset of $T_{\bar{x}(t)}\bigmanif$,
see also \eqref{eq:accTlin-full} in \cref{appendix-lin}: 
\begin{equation}
  \label{eq:accTlin}
  \begin{split}
    \ \hspace{-.9em}
    \AccLin^{U}_{0,t}(\deltax_0)=\{\deltax(t),\; \text{with }
    s\mapsto(\xi(s),\deltau(s),\deltax(s)) &\text{ solution of \eqref{eq:24} 
      on $[0,t]$},
    \\ 
    &\quad\deltax(0)\!=\!\deltax_0,\;
    \deltau([0,t])\!\subset\! U\}.\hspace{-.6em}
  \end{split}
\end{equation}
The following result states
that Condition \eqref{eq:13} is equivalent to controllability of the constrained linear
approximation around $\bar x(.)$ (obviously it also applies to \eqref{eq:10} and \eqref{eq:10bis}).
\begin{lemma}
  \label{lem:lin}
  \begin{enumerate}[label={\textbf{\textup{\Roman*.}}}]
  \item The linear constrained attainable set $\AccLin^U_{0,\tf}(0)$ is a neighborhood of
    the origin in $T_{\bar{x}(\tf)}\bigmanif$ if and only if Condition \eqref{eq:13} holds. 
  \item The set $\pi'(\bar x(\tf))\bigl(\AccLin^U_{0,\tf}(0)\bigr)$ is a neighborhood of
    the origin in $T_{\pi(\bar{x}(\tf))}\manifM$ if and only if Condition \eqref{eq:13bb} holds.
  \end{enumerate}
\end{lemma}
\begin{proof}
  This follows from \cref{prop:lin} for point I and
  \cref{prop:lin-proj} for point II
  in the Appendix, applied with $t_0=0$,
  with \eqref{eq:24} playing the role of \eqref{eq:44}.
  In coordinates, the $k$\textsuperscript{th} column of
  $B(s)$ is the coordinate vector of $\vecfield^k(\bar x(s))$,  
  and the $k$\textsuperscript{th} column of $\phi(0,s)B(s)$ is the
  coordinate vector of
  $\left(\exp(-s\vecfield^0)_*\vecfield^k\right)(\bar x(0))$.
\end{proof}

\begin{proof}[Proof of \cref{thm:xbar}]
  We give a proof of point II.
  For a proof of point I (particular case of point II)
  without reference to the projection $\pi$,
  replace $\manifM$ with $\bigmanif$, $n$ with $d$, $\pi$ with $I\hspace{-.2ex}d$,
  $\left( \pi\!\circ\!\exp(\tf\vecfield^0) \right)'\!\!(\bar x(0))\,
  \bigl(\bar E\bigr)$ with $\bar E$, and
  $\pi'(\bar x(\tf))\bigl(\AccLin^U_{0,\tf}(0)\bigr)$ with
  $\AccLin^U_{0,\tf}(0)$.
  Also,  according to \cref{rem-condloc}, Conditions \eqref{eq:012} and
  \eqref{eq:013} for arbitrarily small  $\varepsilon>0$ imply local controllability, ``with respect to $\pi$'' or
  not; hence we now prove only \eqref{eq:013}.

  According to \cref{lem:lin}, point II, Condition~\eqref{eq:13bb} implies that
  $\pi'(\bar x(\tf))\bigl(\AccLin^U_{0,\tf}\!(0)\bigr)$,
  projection of the accessible set of the linearized system \eqref{eq:24} in time $\tf$,
is a neighborhood of the origin in $T_{\pi(\bar{x}(\tf))}\manifM$.
Let $e_0,\ldots,e_n$ be the vertices of a convex polyhedron in $T_{\pi(\bar{x}(\tf))}\manifM$
that is both contained in $\pi'(\bar x(\tf))\bigl(\AccLin^U_{0,\tf}(0)\bigr)$ and a neighborhood of the origin.
Since $e_0,\ldots,e_n$ are in $\pi'(\bar x(\tf))\bigl(\AccLin^U_{0,\tf}(0)\bigr)$,
there are some $\hat e_0,\ldots,\hat e_n$ in $\AccLin^U_{0,\tf}(0)$
such that $e_i=\pi'(\bar x(\tf))\bigl(\hat e_i\bigr)$,
and there also exist
$\mathfrak{u}_0,\ldots, \mathfrak{u}_n$, some
$U$-valued $\mathrm{L}^\infty_{\mathrm{loc}}$-controls, defined on $[0,\tf]$, that steer the origin to
$\hat e_0,\ldots,\hat e_n$ respectively, in time $\tf$, for the linear time-varying system \eqref{eq:24}.

Going back to the nonlinear system \eqref{sys},
and possibly multiplying the vector fields $\vecfield^0,\ldots,\vecfield^m$  by
  some smooth cut-off function equal to 1 in a neighborhood of
  $\bar x([0,\tf])$ to guarantee that solutions all exist until $\tf$, and in fact
  until $+\infty$, 
let $\mathcal{E}:\mathrm{L}^\infty([0,\tf],\RR^m)\to\bigmanif$
be the end-point mapping at time $\tf$ starting from $\bar x(0)$ for the
nonlinear system \eqref{sys}.
Let also $\mathcal{E}^{\mathrm{lin}}$ be the end-point mapping at time
$\tf$ starting from the origin for the linear system \eqref{eq:24}.
It is well known~\cite{jurdjevic_geometric_1996,agrachev_control_2004} that $\mathcal{E}$ is continuously differentiable and that its
derivative at the zero control is $\mathcal{E}^{\mathrm{lin}}$.
Consider the continuously differentiable map $G:\RR^{n+1}\to\manifM$ defined by
\begin{equation}
  \label{eq:18}
  G(\lambda)=G(\lambda_0,\ldots,\lambda_n)
  =\pi\!\circ\!\mathcal{E}\,\Bigl(\sum_{i=0}^n\lambda_i\mathfrak{u}_i\Bigr)\,.
\end{equation}
Since $\mathcal{E}$ maps the zero control to $\bar x(\tf)$,
and $\frac{\partial G}{\partial\lambda_i}(0,\ldots,0)$, $i\in\{0,\ldots,n\}$,
is equal to
$\pi'(\bar x(\tf))\bigl(\mathcal{E}^{\mathrm{lin}}(\mathfrak{u}_i)\bigr)$
where, by construction,
$\mathcal{E}^{\mathrm{lin}}(\mathfrak{u}_i)\!=\!\hat{e}_i$
and $\pi'(\bar x(\tf))(\hat{e}_i)\!=\!e_i$,
one has
\begin{equation}
  \label{eq:33}
    G(0,\ldots,0)=\pi(\bar x(\tf)) \,,\ \ \;
    \frac{\partial G}{\partial\lambda_i}(0,\ldots,0)=e_i,\,
    \ \;i\in\{0,\ldots,n\}\,.
\end{equation}
The second relation implies that the linear map $G'(0)\!:\RR^{n+1}\!\to
T_{\pi(\bar x(\tf))}\manifM$ satisfies
\begin{equation}
  G'(0)\bigl(\RR^{n+1}_+\bigr)=T_{\pi(\bar x(\tf))}\manifM
\end{equation}
($\RR^{n+1}_+$ is defined in \eqref{eq:orthant})
because $\cone\{e_o,\ldots,e_n\}=T_{\pi(\bar x(\tf))}\manifM$.
Now, applying \cref{lem-open-map} with $k=n+1$, $N=\manifM$,
$W_k=\RR^{n+1}$ and
$V_\varepsilon=\{\lambda\in\RR^{n+1},\,\lambda_0+\cdots+\lambda_n\leq\varepsilon\}$, one
obtains that
$G(\RR^{n+1}_+\cap V_\varepsilon)$ is a neighborhood of  $\pi(\bar x(\tf))$ for any
$\varepsilon$, and this yields \eqref{eq:013} because, since $U$ is convex and contains
the origin, the control $\sum_{i=0}^n\lambda_i\mathfrak{u}_i$ takes values in $\varepsilon
U$ if each $\mathfrak{u}_i$ takes values in $U$, and $(\lambda_0,\ldots,\lambda_n)$ is in
$\RR^{n+1}_+\cap V_\varepsilon$.
\end{proof}

\subsection{Application to periodic solutions of the drift}
\label{sec-conseq}

\Cref{thm:xbar} applies, in particular, to periodic solutions of the drift
vector field under condition~\eqref{per}, yielding, as a direct consequence
(take $\tf=T(x)$, $\bar x(0)=\bar x(\tf)=x$) the following interesting complement to
\cref{thm:CDC-gen}.
\begin{theorem}
  \label{thm:1tour}
  Assume that the vector field $\vecfield^0$ satisfies the periodicity Assumption
  \eqref{per}, and that $U$ is convex compact and satisfies Conditions~\eqref{eq:assU}.
\begin{enumerate}[label={\textbf{\textup{\Roman*.}}}]
\item \label{item-1tour-1}
  If $\cone E^U\!(x)\!=\!T_x\bigmanif$, then
  $\Acc^{\varepsilon U}_{T(x)}(x)$ is a neighborhood of $x$ in $\bigmanif$ for all
  $\varepsilon\hspace{-.08em}>\hspace{-.08em}0$.
\item \label{item-1tour-2}
  Let $\pi$ be as in \cref{thm:CDC-gen}.
  If $\pi'\hspace{-.08em}(x) \left(\cone E^{U}(x)\right)=T_{\pi(x)}\manifM$, then\\
  $\pi\!\left(\Acc^{\varepsilon U}_{T(x)}(x)\right)$ is a neighborhood of $\pi(x)$ in $\manifM$
  for all $\varepsilon>0$. 
\end{enumerate}
  The same conclusions hold for $\Acc^{\varepsilon U}_{-T(x)}(x)$ and
  $\pi\!\left(\Acc^{\varepsilon U}_{-T(x)}(x)\right)$ (backward time) under the same assumptions.
\end{theorem}

This theorem is more precise (locally) than \cref{thm:CDC-gen} in two ways: here,
points in a neighborhood of the initial orbit may be reached over one period without
leaving some neighborhood of the initial orbit 
while \cref{thm:CDC-gen} gives no estimate on final time, and  allows
trajectories to go arbitrarily far.
Since the final time is prescribed, \cref{rem-2} does not apply and parts I and II have no
reason to be equivalent.
Finally, global controllability can be recovered from \cref{thm:1tour} as seen in the proof below.

\begin{proof}[Alternative proof of \cref{thm:CDC-gen}
  (under the assumption that $U$ is convex)] $\ $\\
  According to \cref{thm:1tour}, part II, $\pi\bigl(\Acc^U_{T(y)}(y)\bigr)$
  is a neighborhood of $\pi(y)$ in $\manifM$ for any $y$ in $\bigmanif$,
  and so is $\pi\bigl(\Acc^U_{-T(y)}(y)\bigr)$.
  For $x$ in $\bigmanif$, let us show that $\Acc^U(x)$ is both open and closed in $\bigmanif$;
  this implies the theorem by connectedness of $\bigmanif$.
  First, consider $y$ in  $\Acc^U(x)$, so that
  $\Acc^U(x)$ contains $\Acc^U_{T(y)}(y) $, and
  $\pi\bigl(\Acc^U(x)\bigr)$ contains $\pi\bigl(\Acc^U_{T(y)}(y)\bigr)$, that was just
  pointed out as a neighborhood of $\pi(y)$, hence
  $\pi\bigl(\Acc^U(x) \bigr)$ is a neighborhood of $\pi(y)$, and,
  using \cref{rem-ext} (\cref{sec-sys}),
  $\Acc^U(x)=\pi^{-1}\bigl(\pi\bigl(\Acc^U(x) \bigr) \bigr)$ is a neighborhood of $y$;
  openness is proven.
  Now suppose that some $y$ is in $\overline{\Acc^U(x)}$,
  and $\pi(y)$ in $\overline{\pi\bigl(\Acc^U(x)\bigr)}$.
  We pointed out that $\pi\bigl(\Acc^U_{-T(y)}(y)\bigr)$ is a neighborhood of $\pi(y)$;
  hence it intersects $\pi\bigl(\Acc^U(x)\bigr)$, \emph{i.e.} there is some $z$ such
  that $z\in \Acc^U_{-T(y)}(y)$ and $z'\in \Acc^U(x)$ such that
  $\pi(z)=\pi(z')$, but, according to \cref{rem-ext} again,
  $z'\in \Acc^U(x)$ then implies $z\in \Acc^U(x)$; $z\in \Acc^U(x)$ and $y\in\Acc^U_{T(y)}(z)$  do imply
  $y\in \Acc^U(x)$; this proves closedness and ends the proof of the theorem.
\end{proof}

\subsection{Obstructions to local controllability}
\label{sec-obstruc}
\Cref{thm:CDC-gen,thm:xbar,thm:1tour} state
sufficient conditions (see \eqref{eq:10}, \eqref{eq:10bis}, \eqref{eq:13}, \eqref{eq:13bb},
and the condition in \cref{thm:1tour})
that require some convex cone to be the whole tangent space.
Let us generically call this cone $C$.
The sufficient conditions are not necessary, see, \emph{e.g.}, \cref{exemple-SR},
hence $C$ being contained in a closed half-space is not an
obstruction to controllability.
This section explains
that one gets an obstruction by requiring that $\overline C$
(topological closure) is contained in a closed half-space
and intersects the separating hyperplane (border of the half space) only at the origin.
   Let us recall that, for a cone $C$ in a vector space $V$, the \emph{polar cone} of $C$ is the cone
   $C^\circ\!=\!\{p\!\in\! V^*\!, \,\langle p,v\rangle\!\leq\!0,\,v\!\in\!C\}$;
   a co-vector $p$ in $V^*$ is in the topological interior of $C^\circ$ if and only
   if $\langle p,v\rangle\!<\!0$ for all $v$ in $\overline{C}\setminus\{0\}$.
   Therefore, the property that a cone has a topological closure that intersects some hyperplane
   only at the origin is equivalent to its polar cone having nonempty
   interior.

\Cref{prop:obstruc-xbar} gives an obstruction to local controllability in the
case of \cref{thm:xbar},  \cref{prop:obstruc-per} an
obstruction to the weaker local \emph{orbital} controllability in the
periodic case.
Note that these obstructions to local controllability are \emph{not} obstructions to
global controllability because controllability could be obtained
by ``going far and coming back''. This is a well know fact, noted in
\cref{rem-loc} and illustrated in \cref{exemple-illustration} (on the
  second System \eqref{eq:50}, see the last paragraph of the example).
For conciseness, we do not give the proof of \cref{prop:obstruc-xbar}, to be
deduced from the other proof, \textit{mutatis mutandis}.

\begin{proposition} 
  \label{prop:obstruc-xbar}
  Consider System \eqref{sys}, and $\tf$, $\bar x(.)$, $\bar E$  as in \eqref{eq:014} and
  \eqref{eq:19}.

  \smallskip

  \noindent
  \textup{I.}  If \hspace{1ex}{\bfseries (i)}\ the interior in
  $T^*_{\bar x(0)}\bigmanif$ of the polar cone to $\cone \bar E$ is nonempty, and \hspace{1ex}\\{\bfseries (ii)}\ the vectors
  $\vecfield^1(\bar x(t)\hspace{-.08em})$, \ldots, $\vecfield^m(\bar x(t)\hspace{-.08em})$
  are linearly independent in $T_{\bar x(t)}\bigmanif$ for all $t$ in $[0,\tf]$,
  \hspace{1ex}\ then
  System~\eqref{sys} is not locally controllable around the solution $\bar x(.)$.

  \smallskip

  \noindent
  \textup{II.}
  Let $\pi:\bigmanif\to\manifM$ be as in \cref{thm:xbar}, Part \textup{II}.
   If \hspace{1ex}{\bfseries (i)}\ the interior in $T^*_{\pi(\bar x(\tf))}\manifM$ of the polar cone to
  $\left( \pi\!\circ\!\exp(\tf\vecfield^0)\right)'\!\!(\bar x(0))\!\left(\cone \bar
    E\right)$
  is nonempty, and
  \hspace{1ex}{\bfseries (ii)}\ the vectors $\pi'(\bar x(t))\,\vecfield^1(\bar x(t))$, \ldots, $\pi'(\bar x(t))\,\vecfield^m(\bar x(t))$
  are linearly independent in $T_{\pi(\bar x(t)\hspace{-.08em})}\manifM$,
  \hspace{1ex} then
  System~\eqref{sys} is not locally controllable with respect to $\pi$ around $\bar x(.)$.
\end{proposition}

\begin{proposition} 
  \label{prop:obstruc-per}
  Consider System~\eqref{sys} under Assumption \eqref{per} on $\vecfield^0$ and let
  $\bar x(.)$ be a particular periodic solution of $\dot x=\vecfield^0(x)$ with $\bar
  I=\pi(\bar x(t))$, $t\in\RR$.
  If
  \hspace{1ex}\\{\bfseries (i)}\ the polar cone to $\cone\Bigl(\pi'(\bar x(0))E^U\!(\bar x(0))\Bigr)$
  has a nonempty interior  in $T_{\bar I}\manifM$, and \hspace{1ex}\\{\bfseries (ii)}\ the vectors
  $\pi'(\bar x(t))\vecfield^1(\bar x(t))$, \ldots, $\pi'(\bar x(t))\vecfield^m(\bar x(t))$
  are linearly independent in $T_{\bar I}\manifM$
  for all $t$ in $[0,T(x)]$,
  \hspace{1ex}\\ then System~\eqref{sys} is not locally orbitally controllable around the
  periodic solution $\bar x(.)$. 
\end{proposition}
  In point (i) (\cref{prop:obstruc-per}), $\bar x(0)$ may be replaced by $\bar x(t)$,
  for any other $t$, because
  $\pi'(\bar x(t))E^U\!(\bar x(t))$ does not depend on $t$, see \eqref{eq:2}--\eqref{eq:26a};
this is consistant with the definition of ``locally orbitally controllability'' that
does not take time or initial condition on the orbit into account.
\begin{remark}[Re-formulation of \cref{prop:obstruc-per} in the product form \eqref{sys-Iphi}]
  \label{rmk:reformulation}  
  Since Condition \eqref{per} is assumed and the property is local around a periodic
orbit $\bar I$ (with $\pi(\bar x(t))=\bar I$, $t\in\RR$), one may re-formulate
\cref{prop:obstruc-per} as follows, in the product form \eqref{sys-Iphi},
$(I,\varphi)\in O\times\Sun$ (one may not take $O=\manifM$ in general).
\\
\textit{ 
  If, for some $\bar I$ in $O$,
  \hspace{1ex}\\{\bfseries (i)}\ the polar cone to $\cone\breve{E}^U\!(\bar I\,)$
  has a nonempty interior, and 
  \\{\bfseries (ii)}\ 
  $F^1(\bar I,\varphi),\ldots,F^m(\bar I,\varphi)$
  are linearly independent in $T_{\bar I}\manifM$
  for all $\varphi$ in $\Sun$,
  \hspace{1ex}\\ then System~\eqref{sys-Iphi} is not locally orbitally controllable around the
  periodic solution $\bar I$.
}
\\
The equivalence between the two formulations is easily deduced from \eqref{Xi-Fi}, with
$x=\Phi(I,\varphi)$.
 For local orbital controllability in this setting, see \eqref{eq:401} after
 \cref{def:controllabilty-loc}-III.
\end{remark}
 
\begin{proof}[Proof of \cref{prop:obstruc-per}]
   We prove the re-formulation.
  Let $K=\{u\!\in\!\cone U,\, u_1^{\,2}+$ $\cdots+u_m^{\,2}=1\}$. 
  Fix $\bar p$ in the interior of the polar cone to $\cone\breve{E}^U\!(\bar I\,)$.
  Let $\alpha\!<\!0$ be the maximum of the smooth function
  $(\varphi,u)\!\mapsto\!\langle \bar p,\sum_{k=1}^m u_k F^k(\bar
  I,\varphi)\rangle$ on the compact $\Sun\times K$, where it remains
  negative from points \textit{(i)} and \textit{(ii)}.
  Using some coordinates on an open set $O'\subset O$ around
  $\bar I$, one may define the smooth function $h:O'\to\RR$ by
  $h(I)=\langle \bar p,I\rangle$, and also give a meaning to $\langle \bar p,v\rangle$ if $v\in T_IO$,
  $I\in O'$.
  Then the above implies that there is a neighborhood $O''$
  of $\bar I$ in $O'$ such that
  $\langle \bar p,\sum_{k=1}^m u_k F^k(I,\varphi)\rangle\!<\!\alpha/2\!<\!0$
  for all $(I,\varphi,u)$ in $O''\times\Sun\times K$, hence
  $\langle \bar p,\sum_{k=1}^m u_k F^k(I,\varphi)\rangle\!\leq\!0$
  for all $(I,\varphi,u)$ in $O''\times\Sun\times U$.
  This means that, if $t\mapsto (I(t),\varphi(t),u(t))$ is a
  solution of \eqref{sys-Iphi} defined on $[0,\tf]$ such that $I(0)=\bar I$
  and $I(t)\in O''$ for all $t$ in $[0,\tf]$, then 
  $\frac{\mathrm{d}}{\mathrm{d}t}h(I(t))\leq 0$ for all $t$, hence
  $h(I(\tf))\leq h(\bar I)$.
  This defeats local orbital controllability, taking
  $\Omega=O''$ in \eqref{eq:401}, 
  because $I(\tf)$ is constrained to belong to
  $\{J\in O'', \,h(J)\leq h(\bar{I})\}$, which is not a neighborhood
  of $\bar I$ because $h$ is nondegenerate at $\bar I$, where its derivative is $\bar p\neq0$.
\end{proof}

\newcommand{\thickthicklines}{\linethickness{1.3pt}}
\begin{remark}
  \label{rmk:cond-suppl-neg}
  \setlength{\unitlength}{.0031\textwidth}
  \begin{figure}[h]
    \label{fig:exemple}
    \centering
    \begin{picture}(70,70) 
      \put(-5,20){\vector(1,0){75}}  \put(64,14){$v_1$}
      \put(30,7){\vector(0,1){60.5}}  \put(21,67){$v_2$}
  %
      \thickthicklines
      \put(30,20){\line(1,1){30}}
      \put(30,20){\line(-1,1){30}}
      \thinlines\textcolor{gray}{
        \put(-2,50){\line(1,0){35}}  \put(37,50){\line(1,0){25}}
        \put(-2,57){\line(1,0){64}}}
  %
      \thickthicklines
      \put(0,50){\line(0,1){7}}
      \put(60,50){\line(0,1){7}}
      \put(0,57){\line(1,0){60}}
  %
      \put(33,49){$1$}
      \put(32,60){$1\!+\!I_1$}
      \thinlines\textcolor{gray}{
        \put(0,17){\line(0,1){6}}
        \put(60,17){\line(0,1){6}}
      }
      \put(-3,24){$-1$}
      \put(59,24){$1$}
  %
      \put(20,-2){\makebox(20,5){$I_1>0$}}
    \end{picture}
    \quad
    \begin{picture}(70,70) 
      \put(-5,20){\vector(1,0){75}}  \put(64,14){$v_1$}
      \put(30,7){\vector(0,1){60.5}}  \put(21,67){$v_2$}
  %
      \thickthicklines
      \put(30,20){\line(1,1){30}}
      \put(30,20){\line(-1,1){30}}
      \put(0,50){\line(1,0){60}}
  %
      \put(32,52){$1$}
      \thinlines\textcolor{gray}{
        \put(0,17){\line(0,1){6}}
        \put(60,17){\line(0,1){6}}
      }
      \put(-3,24){$-1$}
      \put(59,24){$1$}
  %
      \put(20,-2){\makebox(20,5){$I_1=0$}}
    \end{picture}
    \quad
    \begin{picture}(70,70) 
      \put(-5,20){\vector(1,0){75}}  \put(64,14){$v_1$}
      \put(30,7){\vector(0,1){60.5}}  \put(21,67){$v_2$}
  %
      \thickthicklines
      \thinlines\textcolor{gray}{
        \put(-2,50){\line(1,0){64}}
        \put(-2,40){\line(1,0){34}} \put(51,40){\line(1,0){11}}
      }
      \thickthicklines
      \put(0,50){\line(1,0){60}}
  %
      \thickthicklines
      \put(0,50){\line(0,-1){10}}
      \put(60,50){\line(0,-1){10}}
      \put(30,10){\line(1,1){30}}
      \put(30,10){\line(-1,1){30}}
  %
      \put(32,52){$1$}
      \put(32,39){$1\!+\!I_1$}
      \thinlines\textcolor{gray}{
        \put(27,10){\line(1,0){6}}
        \put(0,17){\line(0,1){6}}
        \put(60,17){\line(0,1){6}}
      }
      \put(34,8){$I_1$}
      \put(-4.5,24){$-1$}
      \put(58.4,24){$1$}
  %
      \put(20,-2){\makebox(20,5){$I_1<0$}}
    \end{picture}
    \caption{
        For the example in \cref{rmk:cond-suppl-neg}, the set
        $\breve{E}^U\!(I_1,I_2)\subset T_{(I_1,I_2)}\manifM$ defined
        in \eqref{eq:26a} reads
        $\{(u_1\sin\varphi, u_1+u_2\,I_1),\,(u_1,u_2)\in[0,1]^2,\,\varphi\in\Sun\}$
        in the coordinates $v_1,v_2$ associated to $I_1,I_2$.  This
        yields the above representation of $\breve{E}^U\!(I_1,I_2)$,
        depending on the sign of $I_1$.
    }
  \end{figure}
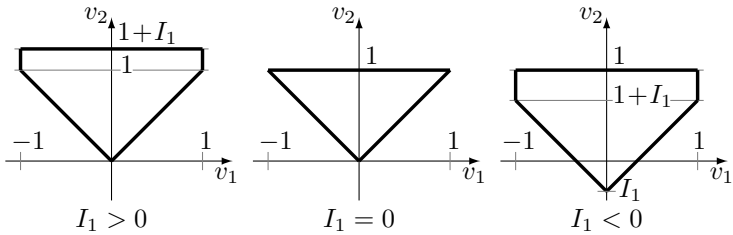
  In \cref{prop:obstruc-per,prop:obstruc-xbar}, it may seem
    odd that the control
  vector fields have to be linearly independent, but it 
  is important in the proof above: if not, at the third line of the proof,
  $\langle\bar p,\sum_{k=1}^m u_k F^k(\bar I,\varphi)\rangle$ could very well vanish at some points in $\Sun\times K$,
  and one could even not prove that the interior of the polar
    cone to $\cone\bigl(\breve{E}^U\!(I\,)\bigr)$ is nonempty for $I$ close
    enough to $\bar I$.
    Let us show that, indeed, Assumption
    \textit{(ii)} cannot be removed from
    \cref{prop:obstruc-per} or \ref{prop:obstruc-xbar}.
  As a counterexample, consider the system of the form \eqref{sys-Iphi} with $m\!=\!n\!=\!2$, $\manifM\!=\!\RR^2$,
  $U\!=\![0,1]\!\times\![0,1]$,
  \begin{displaymath}
    F^1=\partial/\partial{I_2}+\sin\varphi\,\partial/\partial{I_1}
    \quad\text{and}\quad
    F^2=I_1\,\partial/\partial{I_2}\,.
  \end{displaymath}
    Comparing \cref{fig:exemple}, the polar cone at the origin
      to $\breve{E}^U\!(I_1,I_2)$ is $\{0\}$ if
  $I_1<0$ and contains, for instance, $(0,-1)$ in its interior if $I_1\geq0$.
Hence, at  $\bar I=(0,0)$,
\\- Assumption \textit{(i)} of \Cref{prop:obstruc-per} is
  met, but \textit{(ii)} is not, because $F_2(0,0)=0$,
\\- local controllability 
  holds (to reach a neighborhood, first go to close-by
  points where $I_1<0$), so that the conclusion of \Cref{prop:obstruc-per} does not hold,
\\-  the topological interior of the polar cone at the origin to
  $\breve{E}^U\!(I)$ is non-empty at $I=(0,0)$ but empty at some
  arbitrarily close points $(I_1,I_2)$ with $I_1<0$.
%
\end{remark}

\section{Comments, examples and counter-examples}
\label{sec-comments+examples}

\paragraph{On constructively checking the conditions of the theorems}
  Conditions \eqref{eq:10}, \eqref{eq:10bis}, \eqref{eq:13}, \eqref{eq:13bb}, or
  the sufficient condition in \cref{thm:1tour}, involve both the vector fields and
  the control set $U$ and are much more difficult to check than the rank of a
  family of vector fields. 
  This point is discussed by the authors in \cite{JGCD}
  (or in the earlier reference \cite{ecc-2022}),
  where the condition is
  formulated in terms of convex optimization and conveniently solved, in the case of solar
  sails, using sum of square positivity criteria on trigonometric polynomials (this last
  point cannot be generalized to arbitrary periodic systems and sets $U$),
  yielding minimum
  requirements on the characteristics of a solar sail to provide controllability around a
  particular orbit.

\paragraph{Implications of our conditions in terms of Lie brackets}

  The sufficient conditions given in \cref{sec-jurdj-results,sec-xbar,sec-conseq} are based
  on the \emph{conic hull} of the subsets of $T_x\bigmanif$, for some or all $x$ in $\bigmanif$,
  defined in \eqref{eq:1},
  \eqref{eq:EUt1t2} and \eqref{eq:2} being the whole tangent space or not.
  Classical controllability results are not based on cones, but rather on some
  \emph{vector subspaces} of the tangent space,
  generally spanned by collections of Lie Brackets, see \cref{section:state-of-the-art}.
  Obviously, two different subsets of the tangent space may span the same vector subspace
  while generating different convex cones, so that the information contained in the convex
  hull is finer.
  For instance, $\cone E^U_{[t_1,t_2]}(x)$ (or $\cone E^U_\tau$, or $\cone E^U$) a priori depends on $U$ (see
  \cref{exemple-illustration} below), whereas $\vect E^U_{[t_1,t_2]}(x)$ does not
  because, under Assumption~\eqref{eq:assU}, one has  $\vect E^U_{[t_1,t_2]}(x)=\vect E^{\RR^m}_{[t_1,t_2]}(x)$.
  It is however interesting to characterize $\vect E^U_{[t_1,t_2]}(x)$,
both for the sake of
comparison with known results, and because a \emph{necessary} condition for
$\cone E^U_{[t_1,t_2]}(x)=T_x\bigmanif$ is $\vect E^U_{[t_1,t_2]}(x)=T_x\bigmanif$; we will use
this fact in examples to show that our sufficient conditions are
\emph{not} met.

\begin{proposition}\label{prop-crochets}
  With $\mathcal{F}_0$ defined in \eqref{eq:3},
  one has, for any $x$ and any $t_1<t_2$,
  $\vect E^U_{[t_1,t_2]}(x) \supset \vect\mathcal{F}_0(x)$,
  and, under the additionnal assumption that the vector fields
  $\vecfield^0,\ldots,\vecfield^m$ are real analytic,
  $\vect
  E^U_{[t_1,t_2]}(x) = \vect\mathcal{F}_0(x)$.
\end{proposition}

\noindent
Let us give a short proof, although it is classical:
introducing properties of families of \emph{analytic} vector fields in control dates back
to the work of Hermann and Krener in the early 1960s, see
\cite[Chapter 2, Theorem 6]{jurdjevic_geometric_1996}.
Assuming that $\vecfield^0$ is complete, set
\begin{align}
  \label{eq:4}
  \mathcal{F}_{(\alpha_1,\alpha_2)}&=\{\exp(-t\vecfield^0)_*\vecfield^k,\,
                 k\in\{1,\ldots,m\},\;t\in\RR,\,\alpha_1<t<\alpha_2\},\ 
                      \alpha_1<\alpha_2, \\
  \label{eq:6}
  \mathcal{F}_\infty&=\{\exp(-t\vecfield^0)_*\vecfield^k,\,
                 k\in\{1,\ldots,m\},\,t\in\RR\}\,.
\end{align}

\begin{lemma}
  \label{lem-ana}
  One has
  $\vect\mathcal{F}_0(x)
  \subset\vect\mathcal{F}_{(\alpha_1,\alpha_2)}(x)\subset\vect\mathcal{F}_\infty(x)$,
  for any $\alpha_1<\alpha_2$, and even,
  under the additionnal assumption that the vector fields
  $\vecfield^0,\ldots,\vecfield^m$ are real analytic, 
  $\vect\mathcal{F}_0(x) =\vect\mathcal{F}_{(\alpha_1,\alpha_2)}(x) =\vect\mathcal{F}_\infty(x)$.
\end{lemma}
\begin{proof} For $x$ in $\bigmanif$, $k$ in $\{1,\ldots,m\}$ and $p$ an element of
  $T_x^*\bigmanif$, define the smooth map $a_k\!:\RR\to\RR$ by
  $a_k(t)=\left\langle p,
      \exp(-t\vecfield^0)_*\vecfield^k\,(x)\right\rangle$.
  Since $\frac{\mathrm{d}}{\mathrm{d}t} \left(\exp(-t Y)_* Z\right)(x)
    =
    \left(\exp(-t Y)_* [Y,Z]\right)(x)$
    for any vector fields $Y$ and $Z$, one has, for all $t$,
  \begin{equation}
    \label{eq:5}
     \frac{\mathrm{d}^j\,a_k}{\mathrm{d}t^j}(t)=\left\langle p\,,\,
      \exp(-t\vecfield^0)_* \,\mathrm{ad}_{\vecfield^0}^j \vecfield^k\,(x)\right\rangle.
  \end{equation}
  If $p$ is in the annihilator of $\vect\mathcal{F}_{(\alpha_1,\alpha_2)}(x)$, then 
  $a_k(t)=0$ for all 
  $k$ in $\{1,\ldots,m\}$ and $t$ in $(\alpha_1,\alpha_2)$;  differentiating $j$ times, taking $t=0$ and
  using \eqref{eq:5} implies
  \\\mbox{$\ $}\hfill
  $\left\langle p, \mathrm{ad}_{\vecfield^0}^j
    \vecfield^k\,(x)\right\rangle=0\,,\
    k\in\{1,\ldots,m\}\,,\ j\in\NN,
    $
    \hfill\mbox{$\ $}\\
  hence $p$  in the annihilator of $\vect\mathcal{F}_0(x)$. This proves
  $\vect\mathcal{F}_0(x)\!\subset\!\vect\mathcal{F}_{(\alpha_1,\alpha_2)}(x)$, while
  $\vect\mathcal{F}_{(\alpha_1,\alpha_2)}(x) \!\subset\!\vect\mathcal{F}_\infty(x)$ is obvious by definition.
  To prove the reverse inclusion, assume $p$ in the
  annihilator of $\vect\mathcal{F}_0(x)$. For each  $k$ in
  $\{1,\ldots,m\}$, $p$ vanishes on all vectors
  $\mathrm{ad}_{\vecfield^0}^j\vecfield^k\,(x)$, $j\in\NN$, and
   \eqref{eq:5} implies
  $\frac{\mathrm{d}^j a_k}{\mathrm{d}t^j}(0)\!=\!0$ for all $j$.
  Since the map $a_k$ is now real analytic, it must be identically
  zero on $\RR$. This proves 
  $\vect\mathcal{F}_\infty(x)\subset\vect\mathcal{F}_0(x)$, hence $\vect\mathcal{F}_0(x) =\vect\mathcal{F}_{(\alpha_1,\alpha_2)}(x) =\vect\mathcal{F}_\infty(x)$.
\end{proof}

\begin{proof}[Proof of \cref{prop-crochets}]
  Apply \cref{lem-ana}  with
  $\alpha_1=t_1\pm\varepsilon$, $\alpha_2=t_2\pm\varepsilon$,
  taking into account the fact that, for any $\varepsilon>0$,
  \par $\ $ \hfill $\displaystyle \vect\mathcal{F}_{(t_1+\varepsilon,t_2-\varepsilon)} \subset
  \vect\!\!\bigcup_{\tau\in[t_1,t_2]}\hspace{-.6em}E^U_\tau\!(x) \subset
  \vect\mathcal{F}_{(t_1-\varepsilon,t_2+\varepsilon)}. $\hfill $\,$
\end{proof}

The sufficient conditions (\eqref{eq:10}, \eqref{eq:10bis}, \eqref{eq:13}, \eqref{eq:13bb}\ldots)
from \cref{thm:CDC-gen,thm:xbar,thm:1tour} now have,
if the vector fields $\vecfield^0,\ldots,\vecfield^m$ are real analytic
(it is more intricate in the $C^\infty$ case),
the following consequences according to \cref{prop-crochets}:
\begin{itemize}
\item Condition \eqref{eq:10} implies
  $\vect(\mathcal{F}_0\cup\{\vecfield^0\})(x)=T_x\bigmanif$ (for
  all $x$); 
\item Condition \eqref{eq:13} implies $\vect\mathcal{F}_0(\bar
  x(0))=T_{\bar x(0)}\bigmanif$
  \\(and $\vect\mathcal{F}_0(\bar x(t))=T_{\bar x(t)}\bigmanif$ for all $t$ by invariance);
\item Condition \eqref{eq:13bb} implies $\left(\pi\circ\exp(\tf\vecfield^0) \right)'\!\!(\bar x(0))\,
  \bigl(\vect\mathcal{F}_0(\bar x(0))\bigr) = T_{\pi(\bar x(\tf))}\manifM$.
\end{itemize}
These bracket conditions on $\mathcal{F}_0$ appear in some
sufficient conditions for controllability, but are by no means sufficient by
themselves, as evidenced by \cref{exemple-illustration}: the system
there always satisfies these Lie bracket conditions but may be
controllable or not according to the value of parameters
that change the shape of $U$.

It is important to note that these bracket conditions are quite more
restrictive than the classical bracket generating condition  discussed in
\cref{section:state-of-the-art} because $\mathcal{F}_0$ discards all Lie
brackets like $[\vecfield^1,[\vecfield^0,\vecfield^1]]$ or $[\vecfield^1,\vecfield^2]$
involving control vector fields more than once, compare \eqref{eq:45}, \eqref{eq:3} and
\cref{def:larc}.
In fact, full rank of $\mathcal{F}_0$ amounts to controllability \emph{of the
linear approximation}; see, \emph{e.g.}, the discussion in \cite[Section~3.2]{Coro07} around
Equation (3.35) for the case where $U$ is a neighborhood of the
origin, or \eqref{eq:24} and \cref{lem:lin} in the
present paper for the case where $U$ is not.
Studying controllability of systems that are bracket generating but do
not satisfy our condition on $\mathcal{F}_0$, while the control is constrained
to a set $U$ that is not a neighborhood of the origin,
is outside the scope of this paper. 
\cref{exemple-SR} displays some systems in this category for which we prove
controllability by ad-hoc methods, whose generalisation could be
a topic of further investigation.

\begin{example}[illustration of the paper's results]
  \label{exemple-illustration}
  Consider the following  system, of the form \eqref{sys-Iphi},
  with $F^1,F^2$ that depend on $\varphi$
  only, $I=(I_1,I_2,I_3)\in O=\manifM=\RR^3$, $\varphi\in\Sun$ and control
  $u=(u_1,u_2)$. 
\begin{gather}
  \label{eq:48}
  \begin{split}
      \begin{pmatrix}
        \dot I_1 \\ \dot I_2 \\ \dot I_3
      \end{pmatrix}
      \!&= \LL(\theta,\varphi,u):=
      \begin{pmatrix}
        \cos\varphi & -\sin\varphi & 0 \\
        \sin\varphi & \cos\varphi & 0 \\
        0 & 0 & 1
      \end{pmatrix}
      \begin{pmatrix}
        \cos\theta & 0& -\sin\theta \\ 0 & 1 & 0  \\ \sin\theta & 0& \cos\theta
      \end{pmatrix}
      \begin{pmatrix}
         u_1\\ 0 \\ u_2
       \end{pmatrix}
      \,,
      \\
           \dot\varphi\ \;&=\, 1\,,
      \\
           u\ \;&\in\,  U=\{(u_1,u_2)\in\RR^2,\;{u_1}^2+{u_2}^2\leq1,\;|u_2|\leq
  u_1\tan\alpha\,\}\,,
      \end{split}
\end{gather}
where $\theta,\alpha$ are fixed parameters,
$-\frac\pi2\!<\!\theta\!<\!\frac\pi2$, $0\!<\!\alpha \!<\! \frac\pi2$,
$0\!<\!|\alpha|\!+\!|\theta|\!<\!\frac\pi2$; the set $U$ is depicted in \cref{fig-a}.
This system is also of the form \eqref{sys} on $\bigmanif\!=\!\RR^3\!\times\!\Sun$.
\begin{figure}[h]  
 	\centering
	\includegraphics[width=.25\textwidth]{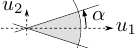}
 	\caption{\label{fig-a} Control set $U\subset\RR^2$ for System~\eqref{eq:48}.}
\end{figure}
  
The definition \eqref{eq:26a} of $\breve{E}^U\!(I)\!\subset\! T_I M\!=\!\RR^3$ yields
$\breve{E}^U\!(I)=\cup_{\varphi\in\Sun}\LL(\theta,\varphi,U)$ for the present system, where
$\LL(\theta,\varphi,U)$ is (in spherical coordinates language) a planar piece in the meridian half plane of
longitude $\varphi$, isometric to $U$ (see \cref{fig-a}), rotated around the origin to bring its axis to lattitude $\theta$;
taking the union over $\varphi$ yields the volume pictured in \cref{fig},
qualitatively different depending whether $|\theta|$ is larger or smaller than
$\alpha$. Analytically, this volume is
$\{v\!\in\!\RR^3,\,\|v\|\leq 1\text{ and }\sqrt{v_1^2\!+\!v_2^2}\tan(\theta-\alpha) \leq v_3\leq
    \sqrt{v_1^2\!+\!v_2^2}\tan(\theta+\alpha)\,\}$.
\begin{figure}[h] 
	\centering
		\begin{minipage}[b]{.4\linewidth}
			\begin{center}
				\includegraphics[width=1\textwidth]{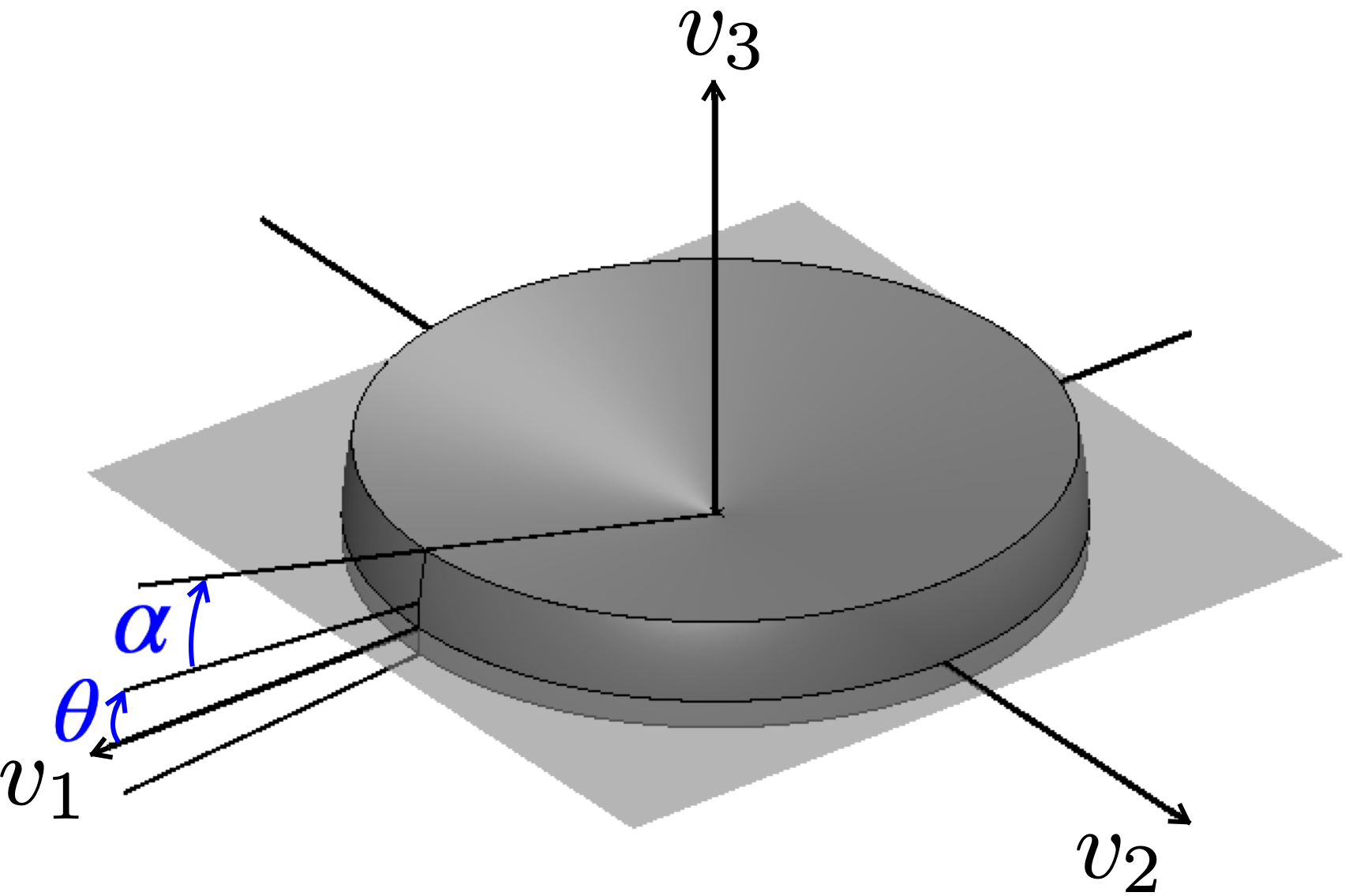}
				\\[0.3em](a) $\breve{E}^U\!(I)$ if $0\!<\!\theta\!<\!\alpha$
			\end{center}
		\end{minipage}
	\hspace{3em} 
		\begin{minipage}[b]{.38\linewidth}
			\begin{center}
                          \includegraphics[width=1\textwidth]{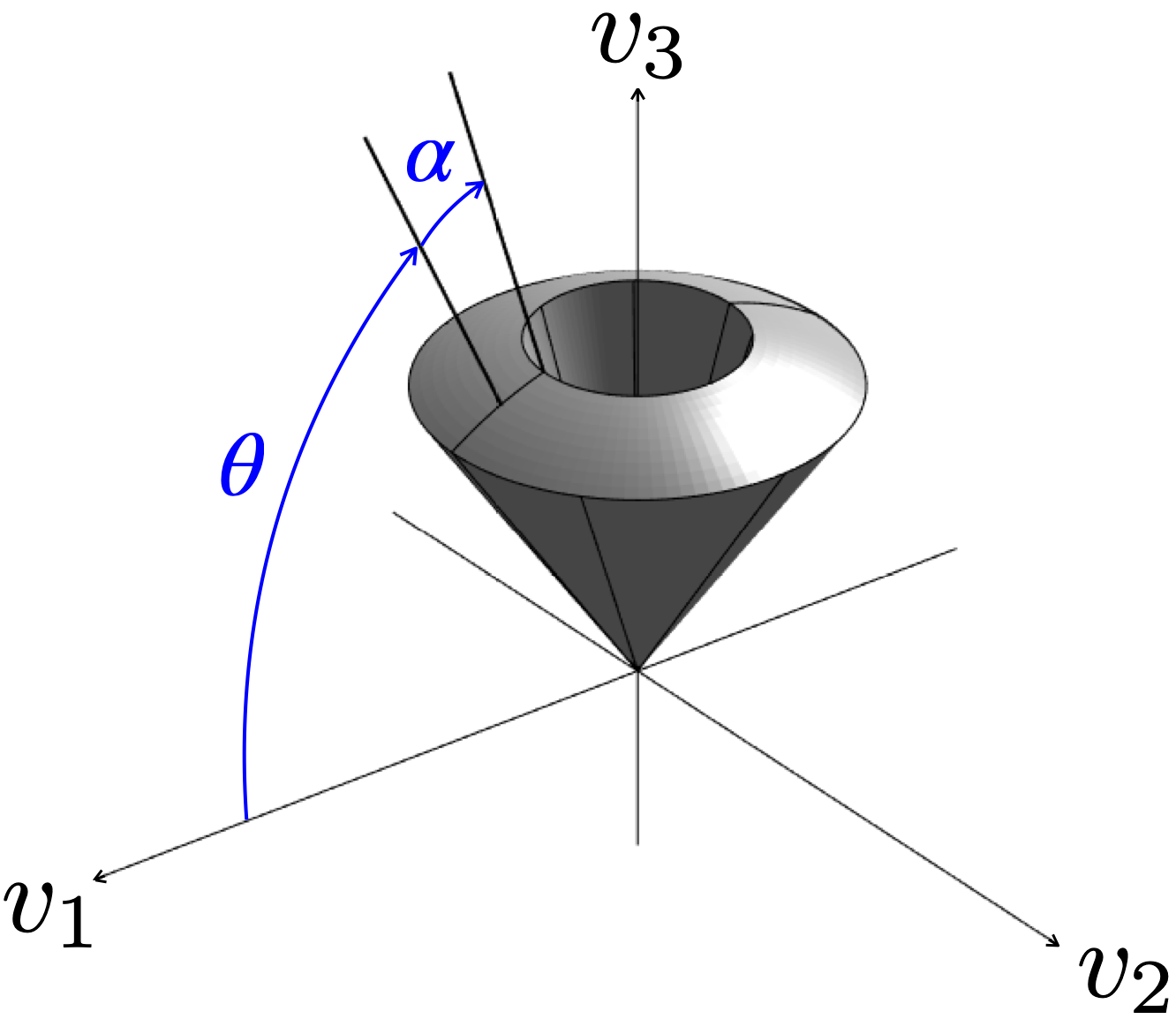}
				\\[-.4em](b) $\breve{E}^U\!(I)$ if $\theta>\alpha$
			\end{center}
		\end{minipage}
	\caption{\label{fig}
		Set $\breve{E}^U\!(I)$ for System~\eqref{eq:48}.
	}
\end{figure}
If $|\theta|<\alpha$, $\breve{E}^U(I)$  has the effect shown
in \cref{fig}-(a), its conic hull is the whole $T_I\manifM$,
Condition~\eqref{eq:10bis} holds, \cref{thm:CDC-gen} yields global controllability and 
\cref{thm:1tour} local controllability over one period.
If $\theta>\alpha$, $\breve{E}^U(I)$ has the effect shown in
\cref{fig}-(b), with a symmetric picture contained in the lower half space
if $\theta<-\alpha$; in both cases, its  dual cone is a cone of
  revolution of half angle $|\theta|-\alpha$, with nonempty interior; as a conclusion,
\cref{prop:obstruc-per} yields non-controllability if $|\theta|>\alpha$.
If $|\theta|=\alpha$, one sees by inspection that the system is not controllable
($\dot I_3$ has a fixed sign), but no result from the paper applies; indeed,
$\cone \breve{E}^U(I)$ is a closed half space, so that
\cref{thm:CDC-gen} (or \ref{thm:1tour}) does not apply, but neither does
\cref{prop:obstruc-per} because the dual cone to $\cone \breve{E}^U(I)$ is only a half line.

From the point of view of the remarks on Lie brackets made just before this example, it is
clear that $\mathcal{F}_0(I,\varphi)$ spans $T_I\manifM$ at all points
(use $F^1,F^2,\partial F^1\!/\partial\varphi,\partial F^2\!/\partial\varphi$), for any value of the parameters:
we clearly see that full rank of $\mathcal{F}_0$ is not, here, sufficient
for controllability, that occurs for some values of the parameters only.

\medskip

Let us now enrich this academic example by making $\theta$ the fourth component $I_4$ of the
state, $\alpha\in[0,\pi/4]$ being still a constant parameter:
\begin{equation}
  \label{eq:50}
    \begin{pmatrix}
      \dot I_1 \\ \dot I_2 \\ \dot I_3
    \end{pmatrix}
    = 
    \LL(I_4,\varphi,(u_1,u_2))
    \,,\ \dot I_4\;=\;u_3\cos^2(2I_4) ,\ \dot\varphi\;=\;1
     ,\ u\in U'\!=U\!\times\![-1,1] ,
\end{equation}
where the control is now $(u_1,u_2,u_3)$ and the state
$(I_1,I_2,I_3,I_4,\varphi)\in\bigmanif=\manifM\times\Sun$, $\manifM=\RR^3\times\left(-\frac\pi4,\frac\pi4\right)$.
The 
factor $\cos^2(2I_4)$ in $\dot I_4$ prevents solutions from leaving
$\bigmanif$
in finite time with admissible controls.
If $\alpha\geq\pi/4$, Condition~\eqref{eq:10bis} holds everywhere and conclusions are the
same as for System \eqref{eq:48} in the case $|\theta|<\alpha$.
If $\alpha<\pi/4$, Condition~\eqref{eq:10bis} holds in the region
$R_1=\{(I_1,I_2,I_3, I_4,\varphi),\;| I_4|<\alpha\}$ and
fails in
$R_0=\{(I_1,I_2,I_3, I_4,\varphi),\; I_4=\pm\alpha\}$ and
$R_2=\{(I_1,I_2,I_3, I_4,\varphi),\;| I_4|>\alpha\}$.
\Cref{thm:CDC-gen} does not apply globally, \cref{thm:1tour} yields local
controllability in one period in region $R_1$, \cref{prop:obstruc-per} tells us that the
system is not locally controllable around its periodic solutions in region $R_2$,
and we see by inspection that this local controllability does not hold either in $R_0$.
Interestingly, this system is also globally
controllable if $\alpha<\pi/4$ although local controllability does not hold in half of the space and
\cref{thm:CDC-gen} does not predict controllability.
Indeed, starting
from some $(I_1^0,I_2^0,I_3^0, I_4^0,\varphi^0)$ to reach
$(I_1^f,I_2^f,I_3^f, I_4^f,\varphi^f)$, one may always use the control
$u_3$ to reach the region $R_1$
(for instance go to $(I_1^0,I_2^0,I_3^0,0,\varphi^0)$), then use the
controls $u_1,u_2$ to reach $I_1^f,I_2^f,I_3^f$ (possible as in \eqref{eq:48}
with $ I_4=0$) with constant $ I_4$ and then use again $u_3$ to
steer $ I_4$ from $0$ to $ I_4^f$.
This displays a case where global controllability holds without local controllability
holding everywhere, hence illustrating a point in \cref{rem-loc}.
\end{example}

\begin{example}[going beyond the paper's results] 
  \label{exemple-SR}
  Consider the following system of the form \eqref{sys-Iphi} with
  scalar control
  ($m=1$, $\bigmanif=\manifM\times\Sun$, $\manifM=\RR^3$, $d=4$):
  \begin{equation}
  \label{eq:52}
  \begin{split}
    &
    \begin{pmatrix}
      \dot{I}_1 \\ \dot{I}_2 \\ \dot{I}_3
    \end{pmatrix}
    =
    \,u \left( \cos\varphi
    \begin{pmatrix}
      1 \\ 0 \\ -I_2/2
    \end{pmatrix}
    +\sin\varphi
    \begin{pmatrix}
      0 \\ 1 \\ I_1/2
    \end{pmatrix}
    \right)\,,
    \\ & \ \ \ \dot\varphi=1\,,
    \hspace{15em}
    u\in U\subset\RR\,.
  \end{split}
\end{equation}
The set $U$ will be either $[-1,1]$ or $[0,1]$.
The system also reads $\dot x=\vecfield^0(x)+u\vecfield^1(x)$, with
$x=(I_1,I_2,I_3,\varphi)$ and the vector fields $\vecfield^0$ and
$\vecfield^1$ defined by:
\begin{equation}
  \label{eq:53}
  \vecfield^0=\frac\partial{\partial\varphi}\,,\quad
  \vecfield^1=\cos\varphi\Bigl(\frac\partial{\partial I_1}-\demi I_2\frac\partial{\partial I_3}\Bigr)
  +\sin\varphi\Bigl(\frac\partial{\partial I_2} +\demi I_1\frac\partial{\partial I_3}\Bigr)\,.
\end{equation}
With reference to the second paragraph of this \cref{sec-comments+examples}, 
the vector fields $\vecfield^0$, $\vecfield^1$, $[\vecfield^0,\vecfield^1]$ and
$[\vecfield^1,[\vecfield^0,\vecfield^1]]$ are linearly independent, 
hence the bracket generating condition is satisfied (this implies global controllability
from \cref{th:BB} if $\, U=[-1,1]$), while
the rank of $\mathcal{F}_0$ is only 2: it is spanned by $\vecfield^1$ and
$[\vecfield^0,\vecfield^1]$ and vanishes on the differential forms
$\demi I_2\mathrm{d}I_1 - \demi I_1\mathrm{d}I_2+\mathrm{d}I_3$ and $\mathrm{d}\varphi$.
We are therefore in the case where our sufficient conditions cannot hold.
  Indeed, a simple computation from \eqref{eq:2}, \eqref{eq:1} and \eqref{eq:26a} yields,
  $E^U(I,\varphi)=\breve{E}^U(I)\times\{0\}$ with
\begin{equation}
  \label{eq:11}
  \breve{E}^U(I)=\{(v_1,v_2,v_3)\in T_I\manifM=\RR^3,\,
  \demi I_2 v_1-\demi I_1 v_2+v_3=0
  \text{ and } v_1^{\,2}+v_2^{\,2}\leq1\}\,,
\end{equation}
whether $U$ is $[-1,1]$ or $[0,1]$ (in the former case, it is ``covered
twice''),
so that $\cone\breve{E}^U\!(I)$ is a proper vector subspace of $T_I\manifM$ and
none of the sufficient conditions in Theorems~\ref{thm:CDC-gen}, \ref{thm:1tour} or
\ref{thm:xbar} hold true, while the polar cone of $\breve{E}^U\!(I)$ is a line in $T^*_I\manifM$, hence has an empty interior,
and one may not either apply \cref{prop:obstruc-xbar} or \cref{prop:obstruc-per} to infer
non-controllability.
No result from the present paper applies here.

Let us however investigate local controllability in the spirit of the conclusions of
\cref{thm:1tour}, by ad-hoc techniques
outside the scope of the results of the paper.
  We are going to prove the following three properties for System~\eqref{eq:52} (recall that the state $x$ is
$(I,\varphi)$, with $I\in\RR^3$, $\pi(I,\varphi)=I$, and the period $T(x)$ is $T(I)=2\pi$):
  \begin{enumerate}[label=(\alph*)]
    \item\label{it-exSR-0} $\pi\bigl(\Acc^{[-\varepsilon,\varepsilon]}_{\,T(I)}(I,\varphi)\bigr)$ is a neighborhood of $I$, 
    \item\label{it-exSR-a} $\pi\bigl(\Acc^{[0,\varepsilon]}_{\,T(I)}(I,\varphi)\bigr)$ \emph{is not} a neighborhood of $I$, 
    \item\label{it-exSR-b} $\pi\bigl(\Acc^{[0,\varepsilon]}_{\,2\,T(I)}(I,\varphi)\bigr)$ is a neighborhood of $I$,
  \end{enumerate}
  for any $\varepsilon$, $0<\varepsilon\leq1$.
Point (a) is in addition of the global controllability mentioned above when $U=[-1,1]$,
while points (b) and (c) concerning the case $U=[0,1]$, more in the spirit of the paper,
display an interesting behavior where local controllability is obtained over two periods but
not over one.
Before establishing them, note that the general solution of \eqref{eq:52} starting from
$(I_1^0,I_2^0,I_3^0,\varphi^0)$ at time zero is given by $\varphi(t)=\varphi^0+t$ and, with
complex notations for $(I_1,I_2)$, 
  \begin{equation}
    \label{eq:28}
    \begin{split}
      I_1(t)+i\, I_2(t)\,&=I_1^0+i\, I_2^0+\int_0^{t }u(s)\,e^{i\,(\varphi^0+s)}\mathrm{d}s\,, \\
      I_3(t)\,&=I_3^0
      +\frac12 \iint_{0\leq r\leq s\leq
        t}\!\!\!\!u(s)\,u(r)\sin(s-r)\,\mathrm{d}r\,\mathrm{d}s
      \\
      & \phantom{=I_3^0}\ 
      +\frac12\,\Im\Bigl(
                     (I_1^0-i\, I_2^0)\,
                     \bigl( I_1(t)-I_1^0+i\, (I_2(t)-I_2^0) \bigr)   
      \Bigr)\,,
    \end{split}
  \end{equation}
For intuition behind equations, recall that $\dot I_3=\frac12(I_1\dot I_2-I_2\dot I_1)$
infers that
the variation $I_3(t_2)-I_3(t_1)$ is the area swept from the origin by the plane curve
$t\mapsto(I_1(t),I_2(t))$ on the interval $[t_1,t_2]$ and
$\varphi$ is (if $u$ is positive) the polar angle of the velocity vector on that 
curve.
It turns out that the area swept from the origin by a closed curve generated by a
non-negative control over a time interval of length at most $2\pi$ cannot be
negative:
    \begin{equation}
    \label{eq:40}
    \Bigl(u(.)\geq0\text{ and } I_1(2\pi)=I_1(0) \text{ and } I_2(2\pi)= I_2(0)\,\Bigr)
    \ \ \Rightarrow\ \ 
    I_3(2\pi)\geq I_3(0)\,.
  \end{equation}
To prove this, first decompose the integral in the expression for $I_3(2\pi)$ in \eqref{eq:28}: 
\begin{equation}
    \label{eq:55}
    \textstyle
    \iint_{0\leq r\leq s\leq 2\pi}=
  \iint_{0\leq r\leq s\leq \pi}+\iint_{\pi\leq r\leq s\leq 2\pi}+\iint_{(s,r)\in[\pi,2\pi]\times[0,\pi]}\,.
\end{equation}
  From \eqref{eq:28}, $(I_1(2\pi),I_2(2\pi))=(I_1^0,I_2^0)$ implies 
  $\int_0^{2\pi}u(s)\sin(s-r)\mathrm{d}s=0$ for any $r$,
  hence $\iint_{(s,r)\in[0,2\pi]\times[0,\pi]}u(s)\,u(r)\sin(s-r) \mathrm{d}r\mathrm{d}s=0$,
  implying that the last term in \eqref{eq:55} is equal to
  $\textstyle-\iint_{(s,r)\in[0,\pi]\times[0,\pi]}
  \sin(s-r)\,u(s)\,u(r)\, \mathrm{d}s\, \mathrm{d}r$,
  which is zero by symmetry around the axis $r=s$;
  hence the integral giving $I_3(2\pi)$ in \eqref{eq:28} reduces to the region
  $\{0\!\leq\! r\!\leq\! s\!\leq\! \pi\}\cup\{\pi\!\leq\! r\!\leq\! s\!\leq\! 2\pi\}$, where
  $u(s)\,u(r)\,\sin(s-r)$ is non negative; this proves \eqref{eq:40}.
  Point~\ref{it-exSR-a} follows.

\begin{figure}    
  \begin{center}
    \hspace{-.04\linewidth}
  \begin{minipage}[b]{.53\linewidth}
    \includegraphics[width=1\textwidth]{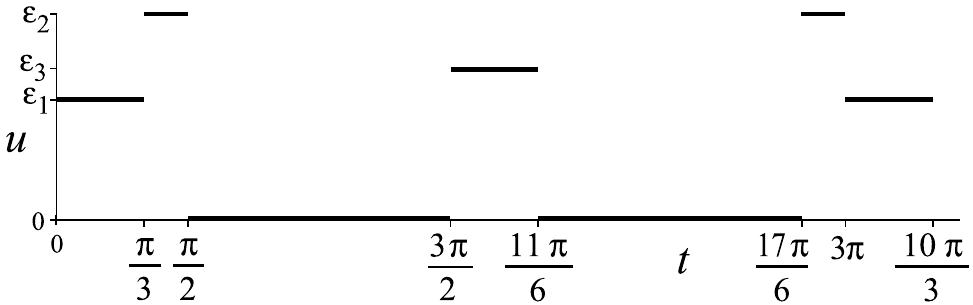}

      \medskip
      
      Control $u(.)$, on $[0,10\pi/3]$\quad{\large$\blacktriangle\blacktriangle$}

      \bigskip
      
      Path followed by $(I_1(t),I_2(t))$ $\ $ {\large$\blacktriangleright\blacktriangleright$}
      \\[-1ex]
  \end{minipage}
  \hspace{-.08\linewidth}
  \begin{minipage}[b]{.49\linewidth}
    \includegraphics[width=1.0\textwidth]{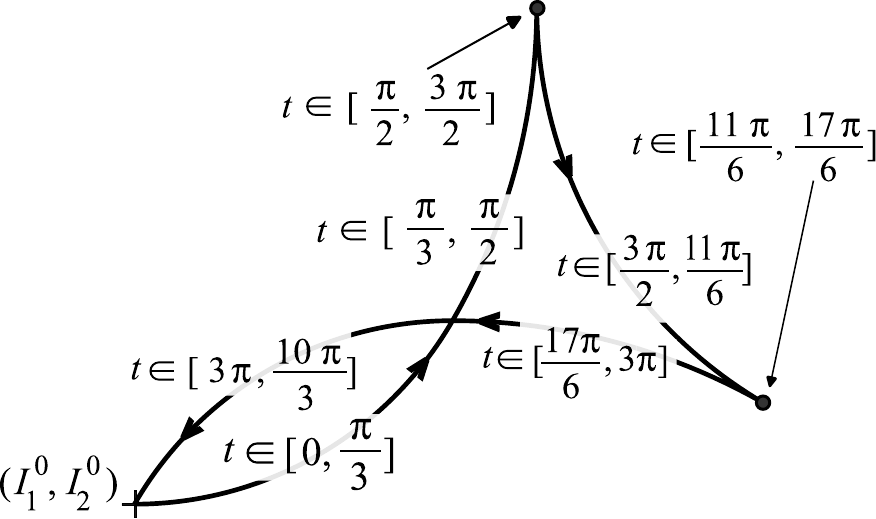}
  \end{minipage}
  \caption{\label{fig2a}
    \emph{Left:} the control used for Point~\ref{it-exSR-0} in \cref{exemple-SR}.
     The constants $\varepsilon_1,
     \varepsilon_2, \varepsilon_3$ are tuned according to
     \eqref{eq:59}, where $\varepsilon_2$ may be taken arbitrarily small.
    \emph{Right:} the closed curve
    $t\mapsto(I_1(t),I_2(t))$, $t\in[0,\frac{10\pi}3]$.
      We drew it from \eqref{eq:28} with
      $u(.)$ given above, $\varepsilon_2=1$ and
      $\varphi(0)=0$.
      Since the curve in the $(I_1,I_2)$-plane scales according to
      $\varepsilon_2$, and the choice of $\varphi(0)$ rotates the
      curve, no axis or units appear on the figure;
      changing $(I_1^0,I_2^0)$ translates the curve and changes the way the $I_3$ component
      of the solution scales with $\varepsilon_2$.
    }
\end{center}
\end{figure}

Let us now turn to Point \ref{it-exSR-b}.
Consider the solution starting from an arbitrary $(I_1^0,I_2^0,I_3^0)$
with the piecewise constant control depicted in
\cref{fig2a} (left) for $t$ in $[0, \frac{10}3\pi]$, continued by zero
on $[\frac{10}3\pi,4\pi]$, with constants $\varepsilon_1,
\varepsilon_2, \varepsilon_3$ chosen such that
\begin{equation}
  \label{eq:59}
  \varepsilon_1=\sqrt{\frac{\left(2 \pi +3\right) \sqrt{3}-5 \pi}{2 \pi -3 \sqrt{3}}}\,\varepsilon_2\,,\ \
  \varepsilon_3=(\sqrt{3}-1)\,\varepsilon_2
  \,,\ \
  0<\varepsilon_1<\varepsilon_3<\varepsilon_2<\varepsilon\,.
\end{equation}
The second relation amounts to
$(I_1(\frac{10\pi}3),I_2(\frac{10\pi}3))=(I_1^0,I_2^0)$,
and the first one to $I_3(\frac{10\pi}3)=I_3^0$
(use formulas \eqref{eq:28})\footnote{
    Once the 2\textsuperscript{nd} condition ensures that the projection on the 
    $(I_1,I_2)$ plane is a closed curve like the one in \cref{fig2a} resembling a goldfish,
    $I_3(\frac{10\pi}3)=I_3^0$ means that the area of the fish's ``body'', run
    counter-clockwise, equals the area of the ``tail'', run clockwise.
    See the two lines after \eqref{eq:28}.
  }.
On the time interval $[0,\frac\pi3]$, the control is constant equal to $\varepsilon_1$,
that belongs to the topological interior of $[0, \varepsilon]$;
the linear approximation on that interval hence reads
$\dot{\deltaI}=A(t) \deltaI+B(t)\deltau$, with no constraint on $\deltau$ and
  $$\textstyle A(t)=\varepsilon_1\left(
  \begin{smallmatrix} 0&0&0 \\ 0&0&0 \\ -\sin\varphi(t) & \cos\varphi(t) &0 \end{smallmatrix}
\right)
\ \text{and}\ 
B(t)=\varepsilon_1\left(
  \begin{smallmatrix} \cos\varphi(t) \\ \sin\varphi(t) \\ I_1(t)\sin\varphi(t) -I_2(t)\cos\varphi(t) \end{smallmatrix}
\right).$$
The vectors
$B(t)$, $(\frac{\mathrm{d}}{\mathrm{d}t}-A(t))B(t)$, $(\frac{\mathrm{d}}{\mathrm{d}t}-A(t))^2B(t)$
are always a basis of $\RR^3$, as seen in the following formulas, where
the $t$ argument is omitted in the right-hand sides
($\varphi$ should be $\varphi(t)$, or $\varphi^0+t$):
$$\textstyle 
(\frac{\mathrm{d}}{\mathrm{d}t}-A(t))B(t)=\varepsilon_1^{\,2}\left(
  \begin{smallmatrix} -\sin\varphi \\ \cos\varphi \\ I_1\cos\varphi +I_2\sin\varphi \end{smallmatrix}\right)
  \!,\ 
(\frac{\mathrm{d}}{\mathrm{d}t}-A(t))^2B(t)=-\varepsilon_1^{\,2} B(t)+\left(
  \begin{smallmatrix} 0 \\ 0 \\ 2\, \varepsilon_1^{\,3}
  \end{smallmatrix}
\right).$$
Controllability of the linear approximation follows, according to \cite{Silv-Mea67} (see
also our remarks after \eqref{eq:32}).
This implies that $\pi\bigl(\Acc^{[0,\varepsilon]}_{\,\pi/3}(I_1^0,I_2^0,I_3^0,\varphi^0)\bigr)$ is a
neighborhood of $(I_1(\frac\pi3),I_2(\frac\pi3),I_3(\frac\pi3))$ according to the classical
linear test for controllability (implicit function theorem applied to the
end-point mapping) because the control set $U=[0,\varepsilon]$ is a neighborhood of the
reference control around which linearization is performed. This implies Point
\ref{it-exSR-b} because  $\Acc^{[0,\varepsilon]}_{\,2\,T(I)}(I^0,\varphi^0)$ contains the
image of $\Acc^{[0,\varepsilon]}_{\,\pi/3}(I^0,\varphi^0)$ by the time-varying flow
corresponding to the reference control from \cref{fig2a} on the time-interval
$[\pi/3,4\pi]$.

  Point~\ref{it-exSR-0} is easier than Point~\ref{it-exSR-b}, negative controls being allowed:
  skip the two intervals or length $\pi$ on which $u\equiv0$ in \cref{fig2a}, left,
  and replace $\varepsilon_3$ with $-\varepsilon_3$;
  the obtained control on $[0,4\pi/3]$ yields the same curve on \cref{fig2a}, right,
  without the stops at each cusp, but with the same linearized controllability.
\end{example}

\appendix

\section{Controllability of time-varying linear systems with constrained control}
\label{appendix-lin}

Consider a time-varying linear control system, with constraints on the
control:
\begin{equation}
  \label{eq:lin}
  \dot{z}  =A(t) z+B(t) v\,,\ \ v\in V\,,
\end{equation}
where the state $z$ belongs to $\RR^d$, $V$ is a compact subset of $\RR^m$, and $t\mapsto A(t)$ and
$t\mapsto B(t)$ are smooth maps $\RR\to\RR^{d\times d}$
and $\RR\to\RR^{d\times m}$ respectively.
This appendix is devoted to giving conditions for
\begin{equation}
  \label{eq:accTlin-full}
  \AccLin^{V}_{t_0,\tf}(z_0)=\{z(\tf),\;
  \text{with }t\mapsto(z(t),v(t))
  \text{ solution of \eqref{eq:lin}},\;
  z(t_0)=z_0\}
\end{equation}
  to be a neighborhood of the origin.
  It is known that $\AccLin^{V}_{t_0,\tf}(z_0)$ is
convex, at least if $V$ is compact, whether $V$ is convex or not, see, \textit{e.g.},
\cite[Chapter 2, Theorem 1A (appendix), p.164]{Lee-Mar86}.

It is customary to define the transition matrix (of $A(.)$) as the map
$(t_1,t_2)\mapsto \phi(t_2,t_1)\in\RR^{d\times d}$ such that,
for any  $t_0\in\RR$ and  $z_0\in\RR^d$, $t\mapsto z(t)=\phi(t,t_0)z_0$
is the solution of
$\dot{z}  =A(t)z$, $z(t_0)=z_0$;
it satisfies
\begin{align}
  \label{eq:31}
  \frac{\partial \phi}{\partial t_2}(t_2,t_1)= A(t_2)\, \phi(t_2,t_1)
  \,,\ \ \ 
  \frac{\partial \phi}{\partial t_1}(t_2,t_1)= -\phi(t_2,t_1)\,A(t_1)
  \,,\ \ \
  \phi(t,t)=I
  \,.
\end{align}
The change of variables $\zeta=\phi(t_0,t)z$
yields, for any $t_0$,
another form of System \eqref{eq:lin}:
\begin{equation}
  \label{eq:44}
  \begin{split}
    &\dot\zeta(t)=\phi(t_0,t)B(t)v(t)\,,\ \ \ \\ & z(t)=\phi(t,t_0)\zeta(t)\,,
  \end{split}
\end{equation}
and the general formula for solutions of \eqref{eq:lin} follows through variation of constants:
\begin{equation}
  \label{eq:35}
  z(t)=\phi(t,t_0)\,z(t_0)+\int_{t_0}^t\phi(t,\tau)B(\tau)v(\tau)\mathrm{d}\tau\,.
\end{equation}

If the convex hull of $V$ is a neighborhood of the origin (in
particular if $V=\RR^m$), one has the following characterization, after \cite[Section 9.2]{Kail80}, \cite{Weis68}, \cite{Silv-Mea67}:
$\AccLin^{V}_{t_0,t}(0)$ is a neighborhood of the origin if and only
if, for $p\in(\RR^d)^*$ (a line vector),
\begin{equation}
  \label{eq:32}
  \left(\vphantom{s^{\frac12}}
    p\,\phi(t_0,s)\,B(s)=0\;\text{for all $s$ in }[t_0,\tf]
  \right)\ \ \Rightarrow\ p=0\,.
\end{equation}
This property is stated in \cite{Kail80} as
``the rows of $s\mapsto \phi(t_0,s)B(s)$ are linearly
independent on the interval $[t_0,\tf]$'', and is equivalent to
  \begin{equation}
    \label{eq:136a}
    \vect\,\{\phi(t_0,s)\,B(s)v,\;s\in[t_0,\tf],\;v\in V\}=\RR^d\,.
  \end{equation}
The case where the convex hull of $V$ is \emph{not} a neighborhood of
the origin, the one of interest here, is treated in \cite{Schm-Bar80}, and sketched in
\cite[Section 2.2]{Lee-Mar86}.
Let us state the precise results needed in \cref{sec-xbar}
(it is difficult to refer to \cite{Schm-Bar80}, that characterizes
    $\cup_{t\in(-\infty,\tf]}\AccLin^{V}_{t,\tf}(0)$ 
    rather than $\AccLin^{V}_{t_0,\tf}(0)$ for a fixed $t_0$).
\begin{proposition}
  \label{prop:lin}
  Assume that $V$ is a compact subset of $\RR^m$ containing zero, and
  that $t_0,\tf$ are two real numbers.
  The accessible set
  $\AccLin^V_{t_0,\tf}(0)$ is a neighborhood of the origin if and only if
  \begin{equation}
    \label{eq:136}
    \cone\,\{\phi(t_0,s)\,B(s)v,\;s\in[t_0,\tf],\;v\in V\}=\RR^d\,.
  \end{equation}
\end{proposition}
\noindent
We also need the following characterization of when some projection of the accessible set
is a neighborhood of the origin in the projection of the state space
(controllability with respect to a
projection, see \cref{def:controllabilty-loc}, part II).
\begin{proposition}
  \label{prop:lin-proj}
  Assume that $V$ is a compact subset of $\RR^m$ containing zero, and
  that $t_0,\tf$ are two real numbers.
  Let $n$ be a positive integer, $n\leq d$,
  and $\ProjLin$ a surjective linear map $\RR^d\to\RR^n$.
  The set $\ProjLin\,\AccLin^V_{t_0,\tf}(0)$ (image of $\AccLin^V_{t_0,\tf}(0)$ by $\ProjLin$)
  is a neighborhood of the origin in $\RR^n$ if and only if
  \begin{equation}
    \label{eq:136b}
    \cone\,\{\ProjLin\,\phi(\tf,s)\,B(s)v,\;s\in[t_0,\tf],\;v\in V\}=\RR^n\,.
  \end{equation}
\end{proposition}
\noindent
Note that, by the separating hyperplane theorem, \eqref{eq:136} and \eqref{eq:136b} are
equivalent to
\begin{align}
  \label{eq:36}
  &\bigl(
      p\,\phi(t_0,s)\,B(s)v\geq0\ \ \text{for all $v\in V$, $s\in[t_0,\tf]$}
    \bigr)  \ \ \Rightarrow\ p=0 \\
  \label{eq:36b}
  \text{and}\ \ \
  &\bigl(
        p\;\ProjLin\,\phi(\tf,s)\,B(s)v\geq0\ \ \text{for all $v\in V$, $s\in[t_0,\tf]$}
    \bigr)
    \ \Rightarrow\ p=0
\end{align}
respectively, $p\in(\RR^d)^*$, $p\in(\RR^n)^*$ in \eqref{eq:36}, \eqref{eq:36b}.
Conditions are given under this form in the literature, rather than
\eqref{eq:136} or \eqref{eq:136b}.

\Cref{prop:lin} is a particular case of \cref{prop:lin-proj}: take $\manifM=\bigmanif$,
$n=d$, and $\ProjLin=I$ (identity map), and factor out the linear isomorphism
$\phi(t_0,\tf)$ in \eqref{eq:136b} to ``replace $\tf$ with $t_0$''.
When $n<d$, the controllability property characterized in \cref{prop:lin-proj} is strictly
weaker than the one in \cref{prop:lin}.
Let us end this appendix with a short proof of \cref{prop:lin-proj}.

\begin{proof}[Proof of \cref{prop:lin-proj}]
  Let us prove necessity first.
  If \eqref{eq:136b} does not hold, there exists a nonzero $p$ such that
  $p\,\ProjLin\phi(\tf,s)\,B(s)v\geq0$ for any $v\in V$ and $s\in[t_0,\tf]$, and
  formula \eqref{eq:35}, with $z_0=0$, obviously implies $p\,\ProjLin z(\tf)\geq0$ for any
  solution,
  implying that $\ProjLin\,\AccLin^V_{t_0,\tf}(0)$ is contained in the half-space $\{y\in\RR^n,\,p\,y\geq0\}$,
  hence is not a neighborhood of $0$.

To prove the converse, assume that $\ProjLin\,\AccLin^V_{t_0,\tf}(0)$ is \emph{not} a neighborhood of 0 in $\RR^n$.
Since $\ProjLin\,\AccLin^V_{t_0,\tf}(0)$ is convex (this was recalled at the beginning of this appendix,
see \cite[Chapter 2, Theorem 1A (appendix), p.~164]{Lee-Mar86}),
it is contained in the half space $\{y\in\RR^n,\,p\,y\geq0\}$, for some nonzero $p\in(\RR^n)^*\setminus\{0\}$, \emph{i.e.} any
solution $(z(.),v(.))$ of \eqref{eq:lin} with $z(t_0)=0$ and $v(.)$ valued in $V$
satisfies $p\ProjLin z(\tf)\geq0$.
It is not too difficult to see that, for the same $p$, one must then have
$p\ProjLin\phi(\tf,s)B(s)v\geq0$ for all $v$ in
$V$ and all $s$ in $[t_0,\tf]$; this defeats
\eqref{eq:36b}, hence \eqref{eq:136b}.
\end{proof}


\end{document}